\documentclass[reqno,11pt]{amsart}

\usepackage{fancybox}

\usepackage{wrapfig}
\usepackage[hang,small,bf]{caption}
\usepackage[subrefformat=parens]{subcaption}
\usepackage{comment}
\usepackage{indentfirst}
\usepackage[top=20mm, bottom=20mm, left=30mm, right=30mm]{geometry}
\usepackage{amscd}
\usepackage{amsmath,amsthm,amssymb}
\usepackage{color}
\usepackage{graphicx}
\usepackage{amsfonts}
\usepackage{cases}
\usepackage{bm}
\usepackage{pb-diagram}
\usepackage{braket}
\usepackage{multirow}
\usepackage{empheq}

\usepackage{tikz} 
\usetikzlibrary{positioning, intersections, calc, arrows.meta,math} 

\numberwithin{equation}{subsection}
\theoremstyle{definition}
\newtheorem{defi}{Definition}[section]

\newtheorem{exa}[defi]{Example}
\newtheorem{lem}[defi]{Lemma}
\newtheorem{propo}[defi]{Proposition}

\newtheorem{theo}[defi]{Theorem}
\newtheorem{coro}[defi]{Corollary}

\newtheorem{conj}[defi]{Conjecture}

\title{Ghost-OSD Method on Numerical Max-plus Algebra}

\author{Yohei Oshida}
\address{}
\email{gan8ponster33@gmail.com \\ mf20022@shibaura-it.ac.jp}

\date{Month, Day, Year}

\begin{document}

\begin{abstract}
In this paper, we introduce a method for reducing the calculation time of matrix exponentiation calculations on numerical max-plus algebra. In particular, we explain Ghost-OSD Method that can quickly calculate the powers of honest matrix A, where $A$ is a square matrix of order $2m+1$ with non-negative integer $m$, and the directed graph $G$ whose weights are the matrix components of $A$ satisfies  $i-2,i-1 \to i \to i+1,i+2$ for each vertice $i$. This method is based on calculation of the location of zero elements of power of a matrix $A$ and theory of pros and cons of existence of cycle of $G$. 


\end{abstract}

\maketitle

\setcounter{tocdepth}{1}
\tableofcontents{}

\section{Introduction}
We assume that readers are familiar with \cite{famous1} or \cite{famous2}.  Let $N$ be non-negative integer. Let $[*]_{i,j}:=[*]_{i-N,j}$ for $i>N$ and $[*]_{i,j}:=[*]_{i,j-N}$ for $j>N$, where $*$ is the matrix. For $r=2,3,\cdots,N$, let $A$ be the element of $\mathbb{R}_{\text{max}}^{N \times N}$ which satisfies as follows.
\begin{equation}\label{Honest}
[A]_{i,j}
\begin{cases}
\neq - \infty & \text{for $i=j+1,j+2,\cdots,j+r$}, \\
= - \infty & \text{otherwise}.
\end{cases}
\end{equation}
We call the above matrix $A$ Honest matrix. Below, we put $r=2$ and $N=2m+1$ for non-negative integer $m$ with $m \geq 2$. Moreover, we assume that the matrix $A$ satisfies $[A]_{j+1,j} > 0$ and  $[A]_{j+2,j} < 0$ for $j=1,2,\cdots,2m+1$.
In general, the computational time for the power calculation of a sufficiently large matrix is enormous. This may also be the case in the case of Honest matrix $A$. However, we confirmed a method for high-speed calculation of the power calculation 
\begin{equation}
X(k):=A^{\otimes (k+1)} \nonumber
\end{equation}
for non-negative integer $k$. It is called Ghost-OSD Method below.\par

\begin{theo}[Ghost-OSD Method]\label{Main}
\it{On the calculation of $X(k)$ for $k=1,2,\cdots,2m,\cdots$, 
we can shorten the number of calculations of $X(k)$ by following steps $(1)$ to $(5)$ below.\par

$(1)$Let us denote 
\begin{equation}
\beta_1(k)=2(m+1)-2k-3 \nonumber 
\end{equation}
and 
\begin{equation}
\beta_2(k)=2(m+1)-k. \nonumber 
\end{equation}
Then, for $k=1,2,\cdots,m-1$, we have
\begin{equation}
[X(k)]_{i,j}=
\begin{cases}
\max \{ [X(k-1)]_{j+1,j} + [A]_{j+1,j} , [X(k-1)]_{j+2,j} + [A]_{j+2,j} \} & \text{if $\beta_1(k)+i \leq j \leq \beta_2(k)+i-2$}, \\
- \infty & \text{otherwise}. \\
\end{cases}
\nonumber 
\end{equation}
for each i.
\par
$(2)$For $k=m,m+1,\cdots,2m-1$, we can compute as follows.
\begin{eqnarray}
[X(k)]_{i,j}=\max \{ [X(k-1)]_{j+1,j} + [A]_{j+1,j} , [X(k-1)]_{j+2,j} + [A]_{j+2,j} \} \nonumber
\end{eqnarray}
\par
$(3)$For $k=2m$, we can compute as follows:
\begin{equation}
[X(k)]_{i,j}=
\begin{cases}
[A]_{1,2m+1}+ \displaystyle \sum_{p=1}^{2m} {[A]}_{p+1,p} & \text{if $i=j$} \\
\max \{ [X(k-1)]_{j+1,j} + [A]_{j+1,j} , [X(k-1)]_{j+2,j} + [A]_{j+2,j} \} & \text{if $i \neq j$}
\end{cases}
\nonumber
\end{equation}
\par
$(4)$Let $\alpha \in \mathbb{Z}_{>1}$ and $\beta$ be the non-negative intger with $0 \leq \beta \leq k$. We assume that
\begin{eqnarray}
(\alpha-1) * (2m+1) -1 \leq k \leq \alpha * (2m+1) -1. \nonumber
\end{eqnarray}
Then, we can compute as follows:
\begin{eqnarray}
[X(k)]_{i,j}=\max \{ [X(k-1)]_{j+1,j} + [A]_{j+1,j} , [X(k-1)]_{j+2,j} + [A]_{j+2,j} \} 
\end{eqnarray}
\par

$(5)$Let $k=\alpha * (2m+1) -1$. If the matrix $X(2m)$ satisfies the apex property
\begin{equation}\label{TopLaw}
[X(2m)]_{i,j} \leq [A]_{1,2m+1}+ \displaystyle \sum_{p=1}^{2m} {[A]}_{p+1,p},
\end{equation}
then we can compute $X(k)$ as follows.
\begin{equation}\label{TopLawC}
[X(k)]_{i,j}
=
\begin{cases}
\alpha * \Biggr ( [A]_{1,2m+1}+ \displaystyle \sum_{p=1}^{2m} {[A]}_{p+1,p}, \Biggr ) & \text{if $i=j$} \\
\max \{ [X(k-1)]_{j+1,j} + [A]_{j+1,j} , [X(k-1)]_{j+2,j} + [A]_{j+2,j} \} & \text{if $i \neq j$}.
\end{cases}
\end{equation}
On the other hands, if the matrix $A$ does not satisfies $(\ref{TopLaw})$, we 
modify $(\ref{TopLawC})$ as follows.
\begin{eqnarray}\label{TopLawCC}
[X(k)]_{i,j}
=
\max \{ [X(k-1)]_{j+1,j} + [A]_{j+1,j} , [X(k-1)]_{j+2,j} + [A]_{j+2,j} \} 
\end{eqnarray}
}
\end{theo}

Although Theorem $\ref{Main}$ appears long and complex, most of it is merely composed of theory on understanding locations of $-\infty$ for matrixs $A$ and theory on maximum weight of cycles of length $2m+1$ of the graph $\mathcal{G}(A)$. In other words, two theories are important in Theorem $\ref{Main}$. And  we were able to confirmed that $X(k)$ could be computed quickly using Theorem $\ref{Main}$ by performing numerical calculations.


\par

As an application of Ghost-OSD Method, for example, the following results exist.

\begin{theo}\label{main}
\it{
Let $A$ be Honest matrixs. Let $B$ be the diagonal matrix which satisfies $[B]_{i,i}=[B]_{j,j}$ for $i \neq j$. Then, for the matrix $D:=B \oplus A$, we have
}
\begin{eqnarray}
D^{\otimes n}=B^{\otimes n} \oplus \Biggr ( \bigoplus_{k=1}^{n-1} ( B^{\otimes n-k} \otimes A^{\otimes k} ) \Biggr ) \oplus A^{\otimes n}. \nonumber
\end{eqnarray}
\end{theo}

In this paper, we mainly prove Theorem $\ref{Main}$ and we introdeuce numerical calculation results by using Theorem $\ref{Main}$. The structure of this paper is as follows. In section 2, at first, here's what we need to prove Theorem $\ref{Main}$. Next, we prove Theorem $\ref{Main}$. In section 3, we introduce the results of numerical calculations using Python on Visual Studio Code. If you want the actual source code used, please contact us.

Unless otherwise described, we denote
\begin{eqnarray}
L=[A]_{1,2m+1}+\sum_{p=1}^{2m} {[A]}_{p+1,p}, \quad \varepsilon = - \infty,
\end{eqnarray}
and if there exists $\oplus$ and $\otimes$ in the formula, we may write $*$ as multiplication.

\section{Proof}

 In this section, at first, here's what we need to prove Theorem $\ref{Main}$. Next, we prove Theorem $\ref{Main}$. First, let's consider the rewriting of the honest matrix $A$. We consider the vector
\begin{equation}\label{vec1}
[u_i(a,b)]_j=
\begin{cases}
\varepsilon & \text{if $j=1,\cdots,i,i+3,\cdots,2m+1$}, \\
a & \text{if $j=i+1,$} \\
b & \text{if $j=i+2.$}
\end{cases}
\end{equation}
and 
\begin{equation}\label{vec2}
[u_{2m}(a,b)]_j=
\begin{cases}
a & \text{if $j=1$}, \\
\varepsilon & \text{if $j=2,\cdots,2m,$} \\
b & \text{if $j=2m+1.$}
\end{cases},\quad 
[u_{2m+1}(a,b)]_j=
\begin{cases}
a & \text{if $j=1$}, \\
b & \text{if $j=2$}, \\
\varepsilon & \text{if $j=3,\cdots,2m+1.$}
\end{cases}
\end{equation}
By using $(\ref{vec1})$ and $(\ref{vec2})$, the matrix $g(A)$ can rewrite as follows.
\begin{eqnarray}\label{vec3}
A=(u_1(a_{1,1},a_{1,2}),u_2(a_{2,1},a_{2,2}),\cdots,u_{2m+1}(a_{2m+1,1},a_{2m+1,2})),
\end{eqnarray}
where $a_{i,j} \neq \varepsilon$ for $i=1,2,\cdots,2m+1$ and $j=1,2$. Note that $a_{1,2},\cdots,a_{2m+1,2} < 0$. Then we want to consider support of matrix as follows.

\begin{defi}[Support of vector and matrix]
\it{
Let $A \in \mathbb{R}^{n \times n}_{\text{max}}$, where $n \in \mathbb{Z}_{>0}$. Let the i-column vector of $A$ be $v_i$. Then, support of $A$ is defined as
}
\begin{eqnarray}
\text{\rm{supp}}(A)=\max_{1 \leq i \leq n} \biggr \{ \# \{ j \mid [v_i]_{j} \neq \varepsilon \}  \biggr \} \nonumber
\end{eqnarray}
\end{defi}

\begin{exa}Let $B$ be the matrix
$
\begin{bmatrix}
3 & \varepsilon & 4 \\
\varepsilon & 1 & -2 \\
\varepsilon & \varepsilon & \varepsilon 
\end{bmatrix}
$
. For matrixs $B$, we have
\begin{eqnarray}\label{Example_1}
\# \{ j \mid [(3,\varepsilon,\varepsilon)^T]_{j} \neq \varepsilon \}
=1,
\nonumber \\
\# \{ j \mid [(\varepsilon,1,\varepsilon)^T]_{j} \neq \varepsilon \}
=1,
\nonumber \\
\# \{ j \mid [(4,-2,\varepsilon)^T]_{j} \neq \varepsilon \}
=2.
\nonumber
\end{eqnarray}
Hence, we have 
\begin{equation}
\text{supp}(B)=2.
\end{equation}
\end{exa}

In general, support of the matrix $*$ is not always possible to satisfy
\begin{equation}
\min_{1 \leq i \leq n} \biggr \{ \# \{ j \mid [v_i]_{j} \neq \varepsilon \}  \biggr \}=\text{supp}(A).
\end{equation}
However, Honest matrix $A$ alway satisfies the above property. In particular, it follows the below result.

\begin{lem}\label{HWC}[Helping with calculation]
\it{
The honest matrix $A$ alway satisfies
}
\begin{eqnarray}\label{vec4}
\text{\rm{supp}}(A)=2.
\end{eqnarray}
\end{lem}

\begin{proof}Since the matrix $A$ can be denoted as $(\ref{vec3})$, we have
\begin{eqnarray}
\{ j \mid [u_{2m}]_j \neq \varepsilon \}=\{1,2m+1\}, \quad \{ j \mid [u_{2m+1}]_j \neq \varepsilon \} =\{1,2\}
\end{eqnarray}
and
\begin{eqnarray}
\{ j \mid [u_i]_j \neq \varepsilon \} =\{i+1,i+2\}
\end{eqnarray}
for $i=1,2,\cdots,2m-1$. Hence, we have
\begin{eqnarray}
2 =\# \{ j \mid [u_1]_j \neq \varepsilon \} \leq \# \{ j \mid [u_2]_j \neq \varepsilon \} \leq \cdots \leq \# \{ j \mid [u_{2m+1}]_j \neq \varepsilon \}= 2 
\end{eqnarray}
Therefore, the equation $(\ref{vec4})$ holds.
\end{proof}

By $(\ref{HWC})$, we can quickly calculate $X(1),\cdots,X(2m)$ in terms of computer calculations. In fact, for $k=1,\cdots,2m$, we can compute $X(k)$ as follows:
\begin{eqnarray}\label{HWC_eq}
&&[X(k)]_{i,j} \nonumber \\
&=&{[X(k-1) \otimes X(0)]}_{i,j} \nonumber \\
&=&{[X(k-1) \otimes A]}_{i,j} \nonumber \\
&=& \max_{1 \leq k \leq 2m+1} [X(k-1)]_{i,k} + [u_j]_k \nonumber \\
&=& \max \{ [X(k-1)]_{i,i+1} + [u_j]_{i+1} , [X(k-1)]_{i,i+2} + [u_j]_{i+2} , \varepsilon \} \nonumber \\
&=& \max \{ [X(k-1)]_{i,i+1} + [u_j]_{i+1} , [X(k-1)]_{i,i+2} + [u_j]_{i+2} \}
\end{eqnarray}
Hence, the equation $(\ref{HWC_eq})$ can be summarized as follows.

\begin{lem}\label{HelpCal}
\it{
Let $(2m+1)+n:=n$. Then, for non-negative integer $k$ with $k \geq 1$, we have
\begin{eqnarray}\label{HWC1}
[X(k)]_{i,j} = \max \{ [X(k-1)]_{i+1,i} + [u_1]_{i+1} , [X(k-1)]_{i+2,i} + [u_1]_{i+2} \}
\end{eqnarray}
In particular, if $[u_j]_{i+1}=[u_j]_{i+2}$, then we have
\begin{eqnarray}\label{HWC2}
[X(k)]_{i,j} = \max \{ [X(k-1)]_{i+1,i} , [X(k-1)]_{i+2,i} \} + [u_j]_{i+1}
\end{eqnarray}
}
\end{lem}

We can confirm that  we can quickly calculate $X(1),\cdots,X(2m)$ and $X(2m+1)$, $X(2m+2)$, $\cdots$ by $(\ref{HWC1})$.
Moreover, as an attempt to reduce the computation time, we consider a calculation method using graph theory.

At First, the graph $\mathcal{G}(A)$ has the following properties

\begin{lem}\label{CPNo}
\it{
Let $m$ be non-negative integer $m \geq 2$. Let $A \in \mathbb{R}^{2m+1 \times 2m+1}_{max}$ be honest matrix. Then, we have
}
\begin{eqnarray}\label{CPNoEq1}
\# P(i,i;1) = \# P(i,i;2) = \cdots = \# P(i,i;m)=0.
\end{eqnarray}
\end{lem}

\begin{proof}
Suppose that ${[A]}_{i,j}=1$ for $(i,j)$ which satisfies ${[A]}_{i,j} \neq \varepsilon$. Note that we can assume this only if we only need to discuss that the diagonal components are not $\varepsilon$. Let us denote 
\begin{eqnarray}\label{basic}
v=([A]_{1,1},[A]_{1,2},\cdots,[A]_{1,2m+1}). \nonumber
\end{eqnarray}
The above equation can rewrite as follows.
\begin{equation}
[v]_j=
\begin{cases}
\varepsilon & \text{if $j=1,2,\cdots,2m-3,2m-2$}, \\
1 & \text{if $j=2m,2m+1$}.
\end{cases}
\end{equation}
Since we have $\{ j \mid [v]_j \neq \varepsilon \}=\{ 2m , 2m+1 \}$ by $(\ref{basic})$, we have 
\begin{eqnarray}
\{ j \mid [v]_j \neq \varepsilon \} \cap \{ j \mid [u_i]_j \neq \varepsilon \} \neq \emptyset \nonumber
\end{eqnarray}
for $i=2m-2,2m-1,2m$. Hence, we only calculate $v \otimes u_{2m-2}(1,1)$, $v \otimes u_{2m-1}(1,1)$ and $v \otimes u_{2m}(1,1)$ when we calculate  $v \otimes A$. The result of calculation of $v \otimes A$ is as follows: 
\begin{equation}
[v \otimes A]_j=
\begin{cases}
\varepsilon & \text{if $j=1,2,\cdots,2m-3,2m+1$}, \\
2 & \text{if $j=2m-2,2m-1,2m$}.
\end{cases}
\end{equation}
By using the above result, we have
\begin{eqnarray}
\{ j \mid [v \otimes A]_j \neq \varepsilon \}=\{ 2m-2,2m-1,2m \}. \nonumber
\end{eqnarray}
Then we have 
\begin{eqnarray}
\{ j \mid [v \otimes A]_j \neq \varepsilon \} \cap \{ j \mid [u_{i}]_j \neq \varepsilon \} \neq \emptyset \nonumber
\end{eqnarray}
for $i=2m-4,2m-3,2m-2,2m-1$. Hence, we obtain the following result.
\begin{equation}
[( v \otimes A ) \otimes A]_j=[v \otimes A^{\otimes 2}]_j=
\begin{cases}
\varepsilon & \text{if $j=1,2,\cdots,2m-4,2m,2m+1$}, \\
3 & \text{if $j=2m-4,2m-3,2m-2,2m-1$}.
\end{cases}
\end{equation}
Let $n:=2m$. From the above discussion, for $k=1,2,\cdots,m-1$, we have 
\begin{eqnarray}
\{ j \mid [v \otimes A^{\otimes k}]_j \neq \varepsilon \} \cap \{ j \mid [u_{i}]_j \neq \varepsilon \} \neq \emptyset \nonumber
\end{eqnarray}
where $i=n-(k-1),n-k,n-(k+1),\cdots,n-2k$, we obtain
\begin{eqnarray}\label{CPNoImportant}
[v \otimes A^{\otimes k}]_j=
\begin{cases}
\varepsilon & \text{if $j=1$}, \\
\varepsilon & \text{if $j \notin \{ 1 , n-(k-1),n-k,n-(k+1),\cdots,n-2k \}$}, \\
k+1 & \text{if $j=n-(k-1),n-k,n-(k+1),\cdots,n-2k$}.
\end{cases}
\end{eqnarray}
If $k=m$, then we have $[v \otimes A^{\otimes m}]_1 \neq \varepsilon$. Hence, we have the following result:
\begin{equation}
[v \otimes A^k]_1=
\begin{cases}
\varepsilon & \text{if $k=1,2,\cdots,m-1$}, \\
k+1 & \text{if $k>m-1$}.
\end{cases}
\end{equation}
Thus, it follows from the above equation that we obtain
\begin{eqnarray}\label{cycleEq2}
[A]_{1,1}=[A^{\otimes 2}]_{1,1}=\cdots=[A^{\otimes m}]_{1,1}=\varepsilon
\end{eqnarray}
Therefore, it follows that $P(1,1;k)=\emptyset$ for $k=1,\cdots,m$.
\par

Now, let  the function
\begin{eqnarray}
F_r : \{ 1 , 2 , \cdots , n \} \times \{ 1 , 2 , \cdots , n \} \to \{ 1 , 2 , \cdots , n \} \times \{ 1 , 2 , \cdots , n \} \nonumber
\end{eqnarray} 
be defined by $F_r(i,j):=(\sigma_r(i),\sigma_r(j))$, where
\begin{equation}
\sigma_r(i)=
\begin{cases}
i - (r - 1)  -1 & \text{if $i - (r - 1) \neq 1$}, \\
2m+1 & \text{if $i - (r - 1)  = 1$}.
\end{cases}
\end{equation}
Let $g(A')$ be the graph $(V,F_2(E))$. Note that the cycle of   vertex $2$ to $2$ of $g(A)$ equals the cycle of   vertex $1$ to $1$ of $g(A')$. 

\begin{figure}[hbtp]
\begin{tikzpicture}[scale=2.5]
\node (v1) [fill=white, draw, circle] at (1,1) {1};
\node (v2) [fill=white, draw, circle] at (0,0) {2};
\node (v3) [fill=white, draw, circle] at (0.5,-1) {3};
\node (v4) [fill=white, draw, circle] at (1.5,-1) {4};
\node (v5) [fill=white, draw, circle] at (2,0) {5};

\draw[->,color=red,>=stealth,thick](v1)--(v2);
\draw[->,>=stealth,thick](v1)--(v3);

\draw[->,>=stealth,thick](v2)--(v3);
\draw[->,color=red,>=stealth,thick](v2)--(v4);

\draw[->,>=stealth,thick](v3)--(v4);
\draw[->,>=stealth,thick](v3)--(v5);

\draw[->,>=stealth,thick](v4)--(v5);
\draw[->,color=red,>=stealth,thick](v4)--(v1);

\draw[->,>=stealth,thick](v5)--(v1);
\draw[->,>=stealth,thick](v5)--(v2);



\node (v6) [fill=white, draw, circle] at (4,1) {$\sigma_1(2)$};
\node (v7) [fill=white, draw, circle] at (3,0) {$\sigma_1(3)$};
\node (v8) [fill=white, draw, circle] at (3.5,-1) {$\sigma_1(4)$};
\node (v9) [fill=white, draw, circle] at (4.5,-1) {$\sigma_1(5)$};
\node (v10) [fill=white, draw, circle] at (5,0) {$\sigma_1(1)$};

\draw[->,>=stealth,thick](v6)--(v7);
\draw[->,color=red,>=stealth,thick](v6)--(v8);

\draw[->,>=stealth,thick](v7)--(v8);
\draw[->,>=stealth,thick](v7)--(v9);

\draw[->,>=stealth,thick](v8)--(v9);
\draw[->,color=red,>=stealth,thick](v8)--(v10);

\draw[->,>=stealth,thick](v9)--(v10);
\draw[->,>=stealth,thick](v9)--(v6);

\draw[->,color=red,>=stealth,thick](v10)--(v6);
\draw[->,>=stealth,thick](v10)--(v7);

\end{tikzpicture}
\end{figure}

By $F_r$ for $r \in [1,2m] \cap \mathbb{Z}$, we have
\begin{eqnarray}
[A']_{1,1}=[(A')^{\otimes 2}]_{1,1}=\cdots=[(A')^{\otimes m}]_{1,1}=\varepsilon. \nonumber
\end{eqnarray}
Hence, we have
\begin{eqnarray}
[A]_{2,2}=[A^{\otimes 2}]_{2,2}=\cdots=[A^{\otimes m}]_{2,2}=\varepsilon. \nonumber 
\end{eqnarray}
Thus, for $j=1,2,\cdots,2m+1$, we have the following result in a similar way.
\begin{eqnarray}
[A]_{j,j}=[A^{\otimes 2}]_{j,j}=\cdots=[A^{\otimes m}]_{j,j}=\varepsilon \nonumber 
\end{eqnarray}
Therefore, the equation $(\ref{CPNoEq1})$ holds.
\end{proof}

From Lemma \ref{CPNo} and equations $(\ref{CPNoImportant})$, we can immediately obtain the following result. What is important in this result is that the number of pair $(i,j)$ and combinations of $i$ and $j$ for which $[X(k)]_{i,j}$ is not $\varepsilon$ can be calculated completely.

\begin{coro}\label{CPNocoro}
\it{
$(1)$For $k=1,2,\cdots,m-1$, we have
\begin{eqnarray}
[A^{\otimes (k+1)}]_{1,2(m+1)-2k-3} = [A^{\otimes (k+1)}]_{1,2(m+1)-k} = \varepsilon
\end{eqnarray}
$(2)$For $k=1,2,\cdots,m-1$, we have
\begin{eqnarray}
\# \{ i \mid [A^{\otimes (k+1)}]_{1,i} \neq \varepsilon  \} = k+2.
\end{eqnarray}
}
\end{coro}

\begin{proof}Let $\beta_1(k)=2(m+1)-2k-3$ and $\beta_2(k)=2(m+1)-k$.\\

At first , we prove that $(1)$ of Corollary $\ref{CPNocoro}$ holds. By $(\ref{CPNoImportant})$, we have
\begin{eqnarray}
[A^{\otimes (k+1)}]_{1,(2m-2k)-1}=[A^{\otimes (k+1)}]_{1,(2m-(k-1))+1}=\varepsilon. \nonumber
\end{eqnarray}
Hence, we obtain
\begin{eqnarray}
[A^{\otimes (k+1)}]_{1,\beta_1(k)}=[v \otimes A^{\otimes (k+1)}]_{1,\beta_2(k)}=\varepsilon \nonumber
\end{eqnarray}
by the following calculation result:
\begin{eqnarray}
(2m-2k)-1&=&2m+(2-2)-2k-1=\beta_1(k). \nonumber \\
(2m-(k-1))+1&=&2m-k+2=\beta_2(k). \nonumber
\end{eqnarray}
\par
Next, we prove that $(2)$ of Corollary $\ref{CPNocoro}$ holds. Since we have $[A^{\otimes (k+1)}]_{1,i}\neq \varepsilon$ for 
\begin{eqnarray}
i=\beta_1(k)+1,\beta_1+2,\cdots,\beta_2(k)-2,\beta_2(k)-1 \nonumber
\end{eqnarray}
and we have
\begin{eqnarray}
[A^{\otimes (k+1)}]_{1,i}=
\begin{cases}
\varepsilon & \text{if $i=1,2,\cdots,\beta_1(k)$}, \\
\varepsilon & \text{if $i=\beta_2(k),\beta_2+1,\cdots,2m+1$}.
\end{cases}
, \nonumber
\end{eqnarray}
we can calculate $\# \{ i \mid [A^{\otimes (k+1)}]_{1,i} \neq \varepsilon  \}$ as follows.
\begin{eqnarray}
&&\# \{ i \mid [A^{\otimes (k+1)}]_{1,i} \neq \varepsilon  \} \nonumber \\
&=&\# \{ \beta_1(k)+1,\beta_1+2,\cdots,\beta_2(k) -2,\beta_2(k) -1 \} \nonumber \\
&=&\# \{ 1,2,\cdots,\beta_2(k) \} - \# \{ 1,2,\cdots,\beta_1(k) \} - \# \{ \beta_2(k) \} \nonumber \\
&=&\beta_2(k) - \beta_1(k) - 1 \nonumber \\
&=&( 2(m+1)-k ) - ( 2(m+1)-2k-3 ) - 1 \nonumber \\
&=&2(m+1)-k - 2(m+1) + 2k + 3 - 1 \nonumber \\
&=&k+2 \nonumber
\end{eqnarray}
\end{proof}

Second,  the graph $\mathcal{G}(A)$ has the following properties

\begin{propo}\label{NHRR}[Shortcut Relation]
\it{
For non-negative integer $m$ with $m \geq 2$, let 
\begin{eqnarray}
\Gamma_m = \sup \{ | \gamma |_{w} \mid \gamma \in P( i , i ; 2m+1) \}
\end{eqnarray}
Then, we have $\Gamma_m=L$. In addition, if we have $\Gamma_m=| \gamma_m |_{w}$, then the edge $\{ i , i+2 \}$ can not be used to construct $\gamma_m$.
}
\end{propo}

\begin{proof}Let $p_1$ be the path defined by
\begin{eqnarray}
((1,3),(3,5),\cdots,(2m,2m+1),(2m+1,2),(2,4),\cdots,(2m-2,2m)(2m,1)). \nonumber
\end{eqnarray}
Let $p_2$ be the path defined by
\begin{eqnarray}
p_2=((1,2),(2,3),\cdots,(2m,2m+1),(2m+1,1)). \nonumber
\end{eqnarray}
Let $\gamma$ be the element of $P(i,i;2m+1)$. To prove Proposition $(\ref{NHRR})$, it is sufficient to show that $\gamma=p_1$ or $\gamma=p_2$ hold. Note that $\{ p_1 , p_2 \} \subset P(i,i;2m+1)$ holds.

We assume that $\gamma_1$ have one or more cycles. Then, it follows from Lemma $\ref{CPNo}$ that $\gamma$ satisfies $| \gamma |_{w} = | \gamma_1 |_{w} + | \gamma_2 |_{w}$, where $\gamma_1$ be the element of $P(i,i;| \gamma_1 |)$ with 
\begin{eqnarray}\label{NHRR_1}
m+1 \leq | \gamma_1 | \leq 2m+1
\end{eqnarray}
and $ \gamma_2$ be the cycle of $g(A)$. By $(\ref{NHRR_1})$, since we can confirm that $\gamma_2$ satisfies $| \gamma_2 |_{1} \leq m$, it follows from Lemma $\ref{CPNo}$ that $\gamma_2$ is not the cycle of $g(A)$. Hence,
it contradicts that  $\gamma$ is the  element of $P(i,i;2m+1)$. Thus, we have $| \gamma |_{w} = | \gamma_1 |_{w}$.

Now, we assume that  some of the edges needed to construct $\gamma$ are form of $\{ i , i+2 \}$, where $2m+2:=1$ and $2m+3:=2$. Then , we have the following result beacuse of jumping over one or more vertices on $\gamma$.
\begin{eqnarray}
m+1 \leq | \gamma |_1 < 2m+1 \nonumber
\end{eqnarray}
It is contradictory that $\gamma$ is the element of $P(i,i;2m+1)$.\par
Hence, $\gamma=p_1$ or $\gamma=p_2$ hold. In particular, we have
\begin{eqnarray}
\Gamma_m
&=& \sup \{ | \gamma |_{w} \mid \gamma \in P(i,i;2m+1) \} \nonumber \\
&=& \sup \{ | \gamma |_{w} \mid \gamma \in \{ p_1 , p_2 \} \} \nonumber \\
&=& \sup \{ | p_1 |_{w} , | p_2 |_{w} \} \nonumber \\
&=&| p_2 |_{w} \nonumber \\
&=&L. \nonumber
\end{eqnarray}
\end{proof}

Let's prove Theorem $\ref{Main}$ which is the main part of this paper

\begin{proof}[Proof of $(1)$ of Theorem $\ref{Main}$]By Corollary $\ref{CPNocoro}$, we have $[A^{\otimes (k+1)}]_{1,j} \neq$ for $j \in [\beta_1(k)+1,\beta_2(k) - 1]$ and we have $[A^{\otimes (k+1)}]_{1,j} = \varepsilon$ for $j \notin [\beta_1(k)+1,\beta_2(k) - 1]$. Let $f_r:[1,2m+1] \to [1,2m+1]$ be the function defined as follows for $r \in [1,2m]$.
\begin{eqnarray}[\text{Rotation of the graph $\mathcal{G}(A)$}]
\quad
f_r(*)=
\begin{cases}
*-r & \text{if $*-r \geq 1$}, \\
(2m+1) -  (*-r) & \text{if $*-r < 1$}.
\end{cases}
\nonumber
\end{eqnarray}
Let $\mathcal{G}(A')$ be the graph in which each vertex $i_0$ of $\mathcal{G}(A)$ is rewritten to $f_r(i_0)$. Then, since the path of vertexs $i_1$ to $i_2$ of $\mathcal{G}(A)$ is same to the path of vertexs $f_r(i_1)$ to $f_r(i_2)$ of $\mathcal{G}(A')$ from the viewpoint that the edge sets can be equated, the necessary and sufficient condition for $[A^{\otimes (k+1)}]_{i,j+(i-1)} \neq \varepsilon$ is $[A^{\otimes (k+1)}]_{1,j} \neq \varepsilon$. Hence, we have
\begin{eqnarray}
[A^{\otimes (k+1)}]_{i,j}=
\begin{cases}
[A^{\otimes k} \otimes A]_{i,j} \neq \varepsilon & \text{if $(i,j) \in [1,2m+1] \times [\beta_1(k) + i ,\beta_2(k) + i - 2]$,} \\
\varepsilon & \text{otherwise.} 
\end{cases}
,
\nonumber
\end{eqnarray}
where we calculate $\beta_1(k)+i$ as $2m+1-(\beta_1(k)+i)$ if $\beta_1(k)+i > 2m+1$ and $\beta_2(k)+i-2$ as $2m+1-(\beta_2(k)+i-2)$ if $\beta_2(k)+i-2 > 2m+1$.\\

In particular, for $(i,j) \in [1,2m+1] \times [\beta_1(k) + i ,\beta_2(k) + i - 2)]$, we have
\begin{equation}
[A^{\otimes (k+1)}]_{i,j}=\max \{ [A^{\otimes k}]_{j+1,j} + [A]_{j+1,j} , [A^{\otimes k}]_{j+2,j} + [A]_{j+2,j} \} \nonumber
\end{equation}
by Lemma $\ref{HelpCal}$.
\end{proof}

\begin{proof}[Proof of $(2)$ and $(4)$ of Theorem $\ref{Main}$]By Lemma $\ref{HelpCal}$, we have
\begin{equation}
[X(k)]_{i,j}=\max \{ [A^{\otimes k}]_{j+1,j} + [A]_{j+1,j} , [A^{\otimes k}]_{j+2,j} + [A]_{j+2,j} \}. \nonumber
\end{equation}
\end{proof}

\begin{proof}[Proof of $(3)$ of Theorem $\ref{Main}$]
By Proposition $\ref{NHRR}$, we have
\begin{equation}
[X(2m)]_{i,i}
=[A^{\otimes (2m+1)}]_{i,i}
=\max \{ | \gamma |_{w} \mid \gamma \in P( i , i ; 2m+1) \}
=L
=[A]_{1,2m+1}+ \displaystyle \sum_{p=1}^{2m} {[A]}_{p+1,p}. \nonumber
\end{equation}
On the other hand, for $i \neq j$, we have
\begin{equation}
[X(2m)]_{i,j}=\max \{ [X(2m-1)]_{j+1,j} + [A]_{j+1,j} , [X(2m-1)]_{j+2,j} + [A]_{j+2,j} \} \nonumber
\end{equation}
by Lemma $\ref{HelpCal}$. Hence, we obtain the  calculation result of $X(2m)$ as follows.
\begin{eqnarray}
[X(2m)]_{i,j}=
\begin{cases}
[A]_{1,2m+1}+ \displaystyle \sum_{p=1}^{2m} {[A]}_{p+1,p} & \text{if $i=j$}, \\
\max \{ [X(2m-1)]_{j+1,j} + [A]_{j+1,j} , [X(2m-1)]_{j+2,j} + [A]_{j+2,j} \} & \text{if $i \neq j$}.
\end{cases}
\nonumber
\end{eqnarray}
\end{proof}

\begin{proof}[Proof of $(5)$ of Theorem $\ref{Main}$] Suppose that $(\ref{TopLaw})$ holds. Then, the matrix $X(2m)$ satisfies 
\begin{equation}
[X(2m)]_{i,j} \leq [A]_{1,2m+1}+ \displaystyle \sum_{p=1}^{2m} {[A]}_{p+1,p} =L= [X(2m)]_{i,i}.
\end{equation}
by Proposition $\ref{NHRR}$. Hence, we have
\begin{eqnarray}\label{TopLaw1}
[A^{\otimes (2m+1)} \otimes A^{\otimes (2m+1)}]_{i,i}
&=&\max_{k} \{ [A^{\otimes (2m+1)}]_{i,k} + [A^{\otimes (2m+1)}]_{k,i} \} \nonumber \\
&=&\max_{k \neq i} \{ 2 * [A^{\otimes (2m+1)}]_{i,i} , [A^{\otimes (2m+1)}]_{i,k} + [A^{\otimes (2m+1)}]_{k,i} \} \nonumber \\
&=&2 * L
\end{eqnarray}
and we have
\begin{eqnarray}\label{TopLaw2}
[A^{\otimes (2m+1)} \otimes A^{\otimes (2m+1)}]_{i,j}
&=&\max_{k} \{ [A^{\otimes (2m+1)}]_{i,k} + [A^{\otimes (2m+1)}]_{k,j} \} \nonumber \\
&\leq& \max_{k} \{ [A]_{i,i} + [A]_{j,j} \} \nonumber \\
&=&2 * L
\end{eqnarray}
for $i \neq j$. Thus, the matrix $[X(k)]_{i,i}$ can be calculated inductively as follows by $(\ref{TopLaw1})$ and $(\ref{TopLaw2})$.
\begin{eqnarray}
[X(k)]_{i,i}
&=&[X(\alpha * (2m+1) -1)]_{i,i} \nonumber \\
&=&[A^{\otimes \alpha * (2m+1)}]_{i,i} \nonumber  \\
&=&[\underbrace{A^{\otimes (2m+1)} \otimes \cdots \otimes [A^{\otimes (2m+1)}}_{\alpha}]_{i,i} \nonumber \\
&=&[A^{\otimes (2m+1)} \otimes A^{\otimes (2m+1)} \otimes \underbrace{A^{\otimes (2m+1)} \otimes \cdots \otimes [A^{\otimes (2m+1)}}_{\alpha-2}]_{i,i} \nonumber \\
&=&\alpha * L \nonumber
\end{eqnarray}
\end{proof}

Finally, we introduce application examples Ghost-OSD METHOD. By Ghost-OSD METHOD, for $n \in \mathbb{Z}_{>1}$, we can compute $X(1)$, $X(2)$, $\cdots$ , $X(n-1)$. Hence, we believe that there is a possibility of shortening the problem solving time when solving problems involving $X(k)=A^{\otimes (k + 1)}$. Thus, for example, it can be applied in
\begin{equation}\label{GOM2}
D^{\otimes n}
=B^{\otimes n} \oplus \Biggr ( \bigoplus_{k=1}^{n-1} ( B^{\otimes n-k} \otimes \color{red} A^{\otimes k} \color{black} ) \Biggr ) \oplus \color{red} A^{\otimes n}  \color{black},
\end{equation}
where $B$ be the diagonal matrix which satisfies $[B]_{i,i}=[B]_{j,j}$ for $i \neq j$ and $D:=B \oplus A$. Note that the matrix $B$ satisfies $[B]_{i,j}=\varepsilon$ for $i \neq j$.

In this case, we can expect a significant reduction in calculation time because \color{blue} not only \color{black} $B^{\otimes n}$ and $B^{\otimes n-k}$ can be treated as a calculation of a diagonal matrix on linear algebra \color{blue} but also \color{black} most component calculations can be completed by substituting value of multiple of $[B]_{1,1}$.

Let's prove that decomposition $(\ref{GOM2})$  holds. Here is one lemma that is necessary for this purpose.

\begin{lem}\label{kakan}\it{
Let $A \in \mathbb{R}^{n \times n}_{\text{max}}$, where $n \in \mathbb{Z}_{>0}$. Then, there holds $A \otimes B = B \otimes A$, where diagonal matrixs $B \in \mathbb{R}^{n \times n}_{\text{max}}$ which satisfies $[B]_{i,i}=[B]_{j,j}$ for $i \neq j$.
}
\end{lem}

\begin{proof}
Since the matrix $B$ satisfies $[B]_{i,j}=\varepsilon$ for $i \neq j$, we can calculate $[A \otimes B]_{i,j}$ as follows.
\begin{eqnarray}\label{kakan-1}
[A \otimes B]_{i,j}
&=&\max_{k \neq j} \{ [A]_{i,j} + [B]_{j,j} , [A]_{i,k} + [B]_{k,j} \} \nonumber \\
&=&\max_{k \neq j} \{ [A]_{i,j} + [B]_{j,j} , \varepsilon \} \nonumber \\
&=&[A]_{i,j}+[B]_{j,j}
\end{eqnarray}
In the same way, we can calculate $[B \otimes A]_{i,j}$ as follows.
\begin{eqnarray}\label{kakan-2}
[B \otimes A]_{i,j}
&=&\max_{k \neq i} \{ [B]_{i,i} + [A]_{i,j} , [B]_{i,k} + [A]_{k,j} \} \nonumber \\
&=&\max_{k \neq i} \{ [B]_{i,i} + [A]_{i,j} , \varepsilon \} \nonumber \\
&=&[B]_{i,i} + [A]_{i,j}
\end{eqnarray}
Hence, we have
\begin{eqnarray}
[A \otimes B]_{i,j}
&=&[A]_{i,j}+[B]_{j,j} \nonumber \\
&=&[A]_{i,j}+[B]_{i,i}
=[B \otimes A]_{i,j}. \nonumber
\end{eqnarray}
by $(\ref{kakan-1})$, $(\ref{kakan-2})$ and $[B]_{i,i}=[B]_{j,j}$. Therefore, $A \otimes B = B \otimes A$ holds.
\end{proof}

Now, we prove based on the mathematical induction that decomposition $(\ref{GOM2})$ holds.

\begin{proof}
By mathematical induction, We show that $(\ref{GOM2})$ holds. If $n=1$, it is obvious that $(\ref{GOM2})$ holds. Suppose that $(\ref{GOM2})$ holds when $n>1$. Then, we have
\begin{eqnarray}
&&D^{\otimes (n+1)} \nonumber \\
&=&D \otimes D^{\otimes n} \nonumber \\
&=&D  \otimes \Biggr ( B^{\otimes n} \oplus \Biggr ( \bigoplus_{k=1}^{n-1} \biggr ( B^{\otimes (n-k)} \otimes  A^{\otimes k} \biggr ) \Biggr ) \oplus A^{\otimes n} \Biggr ) \nonumber \\
&=&(B \oplus A) \otimes \Biggr ( B^{\otimes n} \oplus \Biggr ( \bigoplus_{k=1}^{n-1} \biggr ( B^{\otimes (n-k)} \otimes A^{\otimes k}  \biggr ) \Biggr ) \oplus A^{\otimes n} \Biggr ) \nonumber \\
&=&\Biggr ( B^{\otimes (n+1)} \oplus
\color{red}
\Biggr ( \bigoplus_{k=1}^{n-1} \biggr ( B^{\otimes ((n+1)-k)} \otimes A^{\otimes k} \biggl ) \Biggl ) 
\oplus ( B \otimes A^{\otimes n} ) 
\color{black}
\Biggl )
 \oplus 
 A \otimes \biggr ( B^{\otimes n} \oplus \Biggr ( \bigoplus_{k=1}^{n-1} \biggl ( B^{\otimes (n-k)} \otimes A^{\otimes k}  \biggl ) \Biggl ) \oplus A^{\otimes n} \biggl ) \nonumber  \\
&=&B^{\otimes (n+1)} \oplus \Biggr ( 
\color{red}
\bigoplus_{k=1}^{n} \biggr ( B^{\otimes ((n+1)-k)} \otimes A^{\otimes k} \biggr )
\color{black}
\Biggr )
\oplus 
\color{blue}
\Bigg ( A \otimes \Bigg ( B^{\otimes n} \oplus \Biggr ( \bigoplus_{k=1}^{n-1} \biggr ( B^{\otimes (n-k)} \otimes A^{\otimes k}  \biggr ) \Biggr ) \oplus A^{\otimes n} \Biggr ) \Bigg )
\color{black} . \nonumber
\end{eqnarray}
The blue part of above equation  can calculate as follows.
\begin{eqnarray}
&&A \otimes \Bigg ( B^{\otimes n} \oplus \Biggr ( \bigoplus_{k=1}^{n-1} \biggr ( B^{\otimes (n-k)} \otimes A^{\otimes k}  \biggr ) \Biggr ) \oplus A^{\otimes n} \Biggr ) \nonumber \\
&=&
( A \otimes B^{\otimes n} )  \oplus \Biggr ( \bigoplus_{k=1}^{n-1} \color{red} A \otimes  \biggr ( B^{\otimes (n-k)} \otimes A^{\otimes k}  \biggr ) \color{black} \Biggr ) \oplus A^{\otimes (n+1)}
\nonumber \\
&=&
( A \otimes B^{\otimes n} )  \oplus \Biggr ( \bigoplus_{k=1}^{n-1} \color{red}\biggr ( A \otimes  B^{\otimes (n-k)} \biggr ) \otimes A^{\otimes k} \color{black} \Biggr ) \oplus A^{\otimes (n+1)}
\nonumber \\
&=&
( A \otimes B^{\otimes n} )  \oplus \Biggr ( \bigoplus_{k=1}^{n-1} \color{red} \biggr ( B^{\otimes (n-k)}  \otimes  A \biggr ) \color{black} \otimes A^{\otimes k} \Biggr ) \oplus A^{\otimes (n+1)}
\nonumber \\
&=&
( A \otimes B^{\otimes n} )  \oplus \Biggr ( \bigoplus_{k=1}^{n-1} \color{red} B^{\otimes (n-k)} \otimes  \biggr ( A \otimes A^{\otimes k}  \biggr ) \color{black} \Biggr ) \oplus A^{\otimes (n+1)}
\nonumber \\
&=&
( A \otimes B^{\otimes n} )  \oplus \Biggr ( \bigoplus_{k=1}^{n-1} \biggr ( B^{\otimes (n-k)} \otimes A^{\otimes (k + 1 )} \biggr ) \Biggr ) 
\oplus A^{\otimes (n+1)}
\nonumber
\end{eqnarray}
Hence, we have
\begin{eqnarray}
D^{\otimes (n+1)}
=B^{\otimes (n+1)} \oplus \Biggr ( \bigoplus_{k=1}^{n} \biggr ( B^{\otimes ((n+1)-k)} \otimes A^{\otimes k} \biggr ) \Biggr ) \oplus
\color{blue}
( A \otimes B^{\otimes n} )  \oplus \Biggr ( \bigoplus_{k=1}^{n-1} \biggr ( B^{\otimes (n-k)} \otimes A^{\otimes (k + 1 )} \biggr ) \Biggr ) 
\color{black}
\oplus A^{\otimes (n+1)}. \nonumber
\end{eqnarray}

Since $A \otimes B^{\otimes n}=B^{\otimes n} \otimes A$ holds by Lemma $\ref{kakan}$, the blue part of above equation  can be calculated as follows.
\begin{eqnarray}
\color{red} ( A \otimes B^{\otimes n} ) 
\color{black} \oplus 
\Biggr ( \bigoplus_{k=1}^{n-1} ( B^{\otimes (n-k)} \otimes A^{\otimes (k + 1 )} ) \Biggr )
&=&
\color{red} ( B^{\otimes n} \otimes A ) 
\color{black} \oplus 
\color{blue}
\Biggr ( \bigoplus_{k=1}^{n-1} ( B^{\otimes (n-k)} \otimes A^{\otimes (k + 1 )} ) \Biggr ) \nonumber \\
&=&
( B^{\otimes n} \otimes A )  \oplus 
\color{blue}
\Biggr ( \bigoplus_{k=2}^{n} ( B^{\otimes ((n+1)-k)} \otimes A^{\otimes k} ) \Biggr ) \nonumber \\
&=&
\bigoplus_{k=1}^{n} ( B^{\otimes (n+1)-k)} \otimes A^{\otimes k} ) \nonumber
\end{eqnarray}
Thus, we have
\begin{eqnarray}
&&D^{\otimes (n+1)} \nonumber \\
&=&B^{\otimes (n+1)} \oplus \Biggr ( \bigoplus_{k=1}^{n} ( B^{\otimes ((n+1)-k)} \otimes A^{\otimes k} ) \Biggr ) \oplus \color{red} \Biggr ( \bigoplus_{k=1}^{n} ( B^{\otimes ((n+1)-k)} \otimes A^{\otimes k} ) \Biggr ) \color{black} \oplus A^{\otimes (n+1)} \nonumber \\
&=&
B^{\otimes (n+1)} \oplus \Biggr ( \bigoplus_{k=1}^{n} ( B^{\otimes ((n+1)-k)} \otimes A^{\otimes k} ) \Biggr ) \oplus A^{\otimes (n+1)}. \nonumber
\end{eqnarray}
Therefore, $(\ref{GOM2})$  holds for $n \geq 1$.
\end{proof}

\section{Computation Speed Analysis}

In this section we introduce computation Speed Analysis result for Theorem $\ref{Main}$ and Theorem $\ref{main}$. We performed it using Python. The setting of  Speed Analysis is below. Let $A$ be honest matrix of order $2m+1$, where $m \in \mathbb{Z}_{>1}$. Suppose that components of matrix A are random numbers. The range of it is from $1$ to $50$ or $-50$ to $-1$. Let $B$ be the diagonal matrix which satisfies $[B]_{i,i}=m$ and $D:=B \oplus A$. So, at first, we confirm that Theorem $\ref{Main}$ bring good effects to calculations of $X(k)$ when $m$ increases. Next, we confirm that Theorem $\ref{main}$ can not bring good effects to calculations of $Z(k):=D^{\otimes (k+1)}$ when $m$ increases. Below, the horizontal axis of all the graphs introduced in this section show value of $m$ for square matrixs $A$ or $D$ of order $2m+1$. On other word, for value of $m$, it means that $X(1),\cdots,X(P)$ or $Z(1),\cdots,Z(P)$ is computed, where $P$ is greater than at least $2m$ and it is fixed. Note that even if the value on the horizontal axis moves to the right, $P$ remains fixed.

\subsection{The case of Theorem $\ref{Main}$}

The following two figures show the measurement results of the calculation time of $X(1),\cdots,X(2m)$. The vertical axis shows the processing time of calculations. The red line shows processing time $t_1$ for normal calculations
\begin{equation}
[X(k)]_{i,j}=\max_{1 \leq k \leq 2m+1} \{ [X(k-1)]_{i,k} + [X(0)]_{k,j} \} \nonumber
\end{equation}
Note that we does not use Lemma $\ref{HelpCal}$ for normal calculations. The blue line shows processing time $t_2$ using Theorem $\ref{Main}$. The green line shows the error $t_1-t_2$.
From the two figures below, it can be said that $t_1 - t_2 \sim t_1$ holds in the sense that the green line is along the red line. In addition, we can see that $t_2$ does not increase so much even if the m increases rapidly. What has been said so far does not change even if the upper limit of the value of m is from 20 to 50.

\begin{figure}[ht]
\centering
\begin{minipage}[b]{0.49\columnwidth}
    \centering
    \includegraphics[width=1.1\columnwidth]{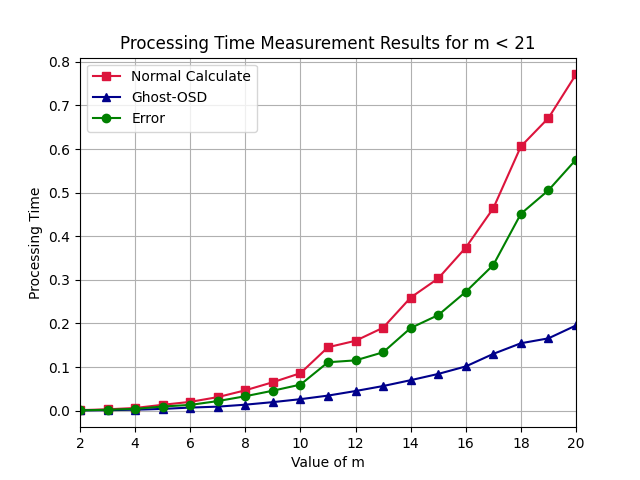}
\end{minipage}
\begin{minipage}[b]{0.49\columnwidth}
    \centering
    \includegraphics[width=1.1\columnwidth]{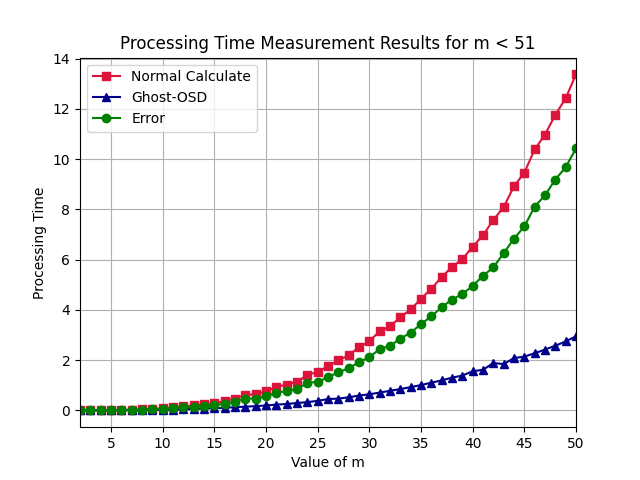}
\end{minipage}
\end{figure}

The above result can be said to be a case where $\beta=1$ for calculations time of $X(0),\cdots,X(\beta(2m+1)-1)$. If $\beta > 1$, for example, if $\beta=10$, then we can obtain the following measurement. It can be confirmed that the same phenomenon as in the case of $\beta=1$ is occurring. From the above numerical analysis results, it can be concluded that Theorem $\ref{Main}$ is highly useful.

\begin{figure}[h]
\centering
\begin{minipage}[b]{0.49\columnwidth}
    \centering
    \includegraphics[width=1.1\columnwidth]{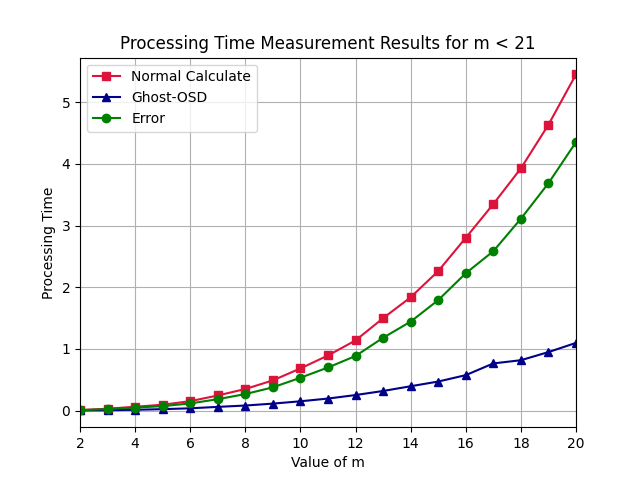}
\end{minipage}
\begin{minipage}[b]{0.49\columnwidth}
    \centering
    \includegraphics[width=1.1\columnwidth]{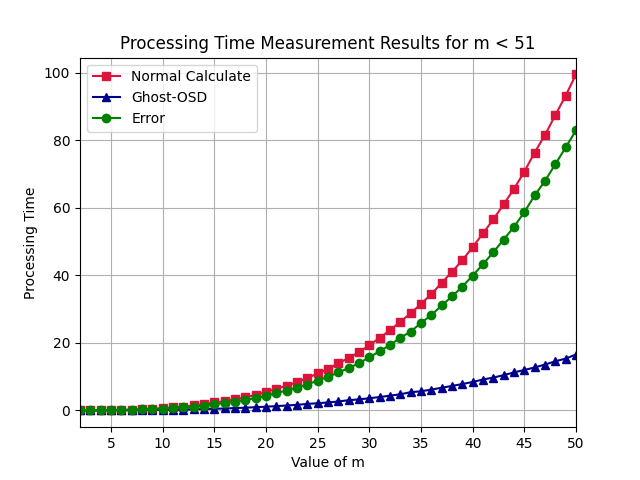}
\end{minipage}
\end{figure}

\subsection{The case of Theorem $\ref{main}$}

The following four figures show the measurement results of the calculation time of $Z(1),\cdots,Z(2m)$. The vertical axis shows the processing time of calculations of $Z(1),\cdots,Z(2m)$ . The red line shows the calculation time $t_1$ for normal calculations. The blue line shows the calculation time $t_2$ using Theorem $\ref{main}$. For $m \in [2,11]$, diverse movements appeared with respect to the crossing and entanglement of the red and blue lines. So, it can be said that  $t_1 \sim t_2$ holds in the sense that the blue line is along the red line very roughly. In this way, it is not possible to judge whether the effect of Theorem 2 is good or bad.

\begin{figure}[h]
\centering
\begin{minipage}[b]{0.49\columnwidth}
    \centering
    \includegraphics[width=1.1\columnwidth]{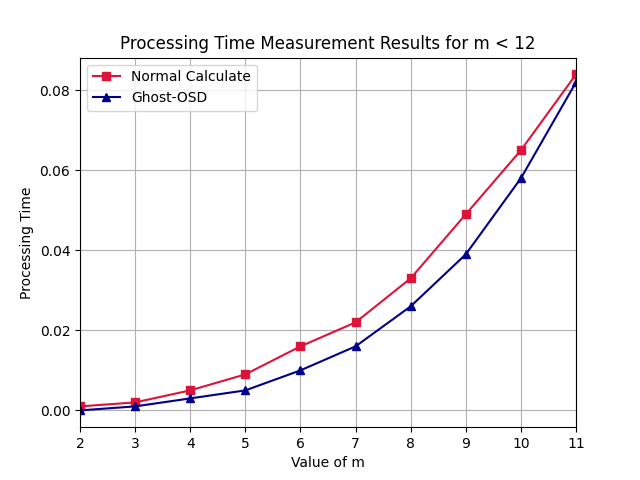}
\end{minipage}
\begin{minipage}[b]{0.49\columnwidth}
    \centering
    \includegraphics[width=1.1\columnwidth]{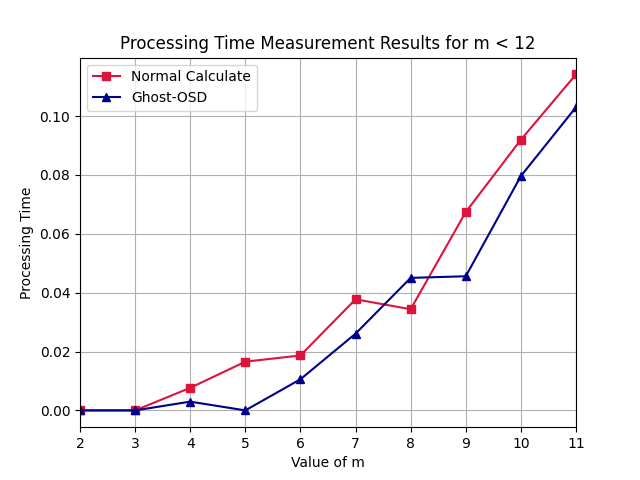}
\end{minipage}
\begin{minipage}[b]{0.49\columnwidth}
    \centering
    \includegraphics[width=1.1\columnwidth]{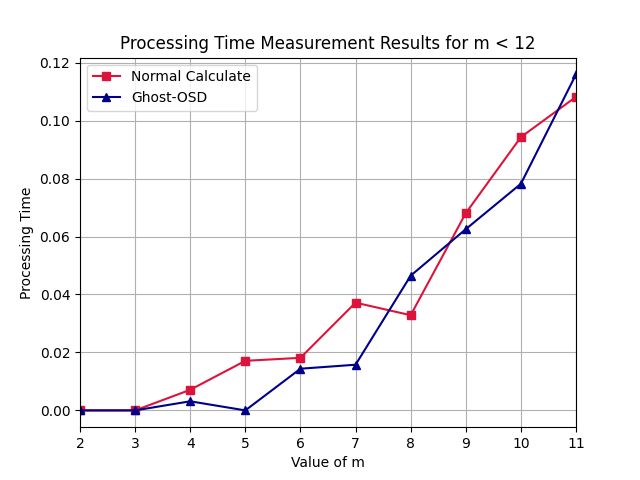}
\end{minipage}
\begin{minipage}[b]{0.49\columnwidth}
    \centering
    \includegraphics[width=1.1\columnwidth]{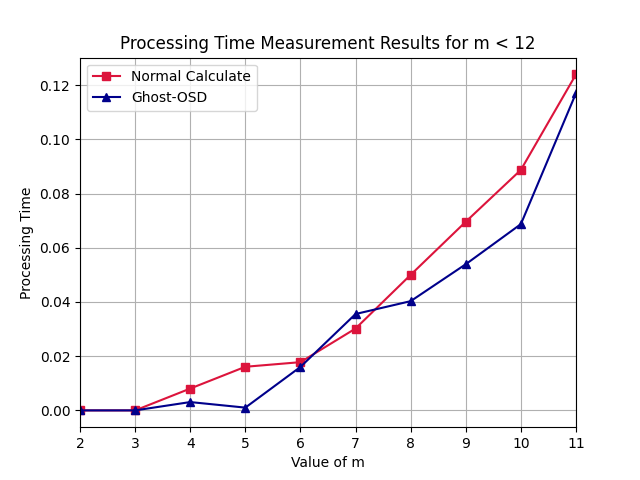}
\end{minipage}
\end{figure}

Let's raise the upper limit of $m$ from $11$ to $20$ or $50$. Then, we have the following result.

\begin{figure}[h]
\centering
\begin{minipage}[b]{0.455\columnwidth}
    \centering
    \includegraphics[width=1.1\columnwidth]{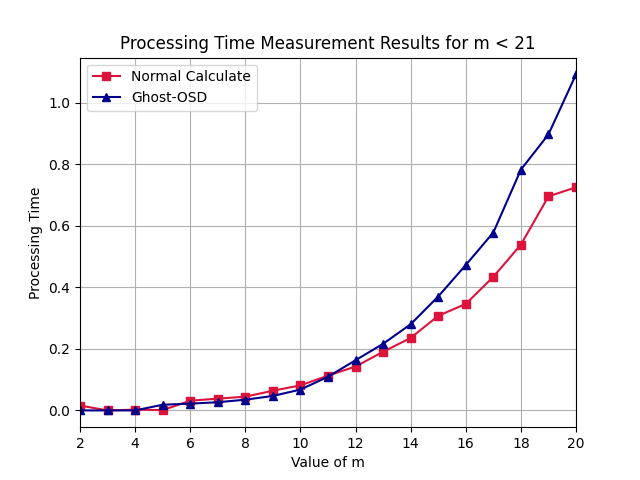}
\end{minipage}
\begin{minipage}[b]{0.455\columnwidth}
    \centering
    \includegraphics[width=1.1\columnwidth]{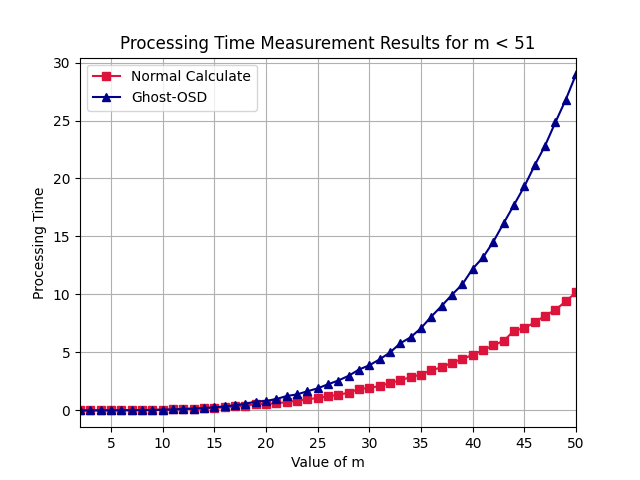}
\end{minipage}
\end{figure}

Since  the blue line is above the red line when $m$ is large enough,
it can be said that  $t_1 < t_2$ holds. In addition, we can see that the blue line is moving farther away from the red line. if the $m$ increases rapidly. From the above numerical analysis results, it can be concluded that Theorem $\ref{main}$ is not highly useful.

In fact, it can be seen from the following two figures that the calculation time is shorter than the normal calculation when the following calculation method is adopted.
\begin{equation}
[Z(k)]_{i,j}
=
\max \{ [Z(k)]_{i,i} + [Z(0)]_{i,i} , [Z(k)]_{i,i+1} + [Z(0)]_{i+1,i} , [Z(k)]_{i,i+2} + [Z(0)]_{i+2,i} \}
\nonumber
\end{equation}
The green line of  the following two figures shows the processing time $t_3$ using the above calculation. 
\begin{figure}[h]
\centering
\begin{minipage}[b]{0.49\columnwidth}
    \centering
    \includegraphics[width=1.1\columnwidth]{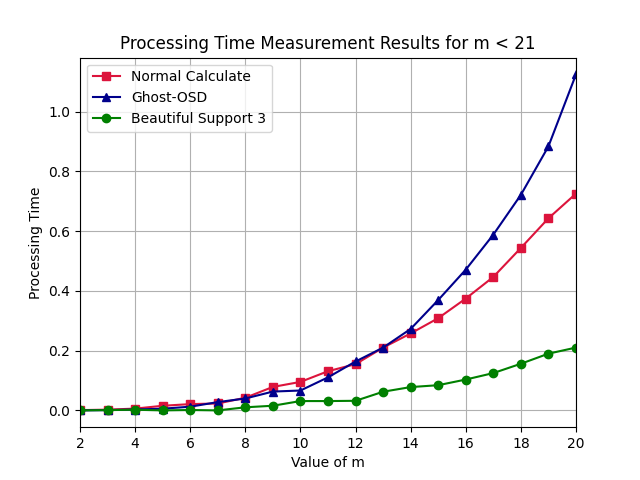}
\end{minipage}
\begin{minipage}[b]{0.49\columnwidth}
    \centering
    \includegraphics[width=1.1\columnwidth]{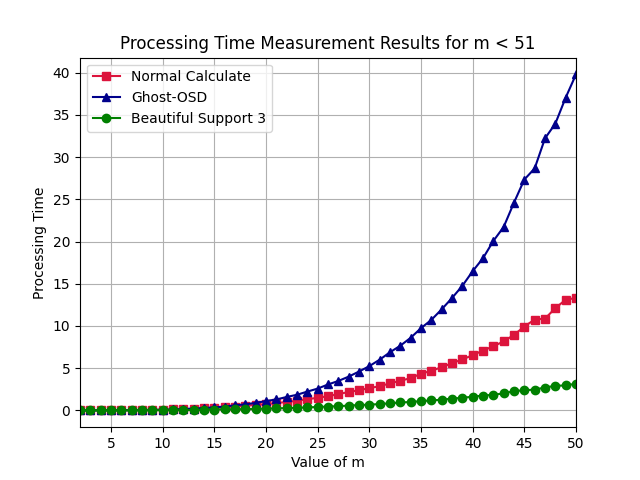}
\end{minipage}
\end{figure}

From above two results, it can be said that $(\ref{error3})$ holds in the sense that the green line is below the red line.
\begin{equation}\label{error3}
t_3 < t_1 < t_2
\end{equation}
In particular, $t_3$ can be seen that the increase does not seem to increase much when $m$ increases. Above all, with the introduction of the green line, we realized that the compression of the number of calculations $\varepsilon + a$ for $a \in \mathbb{R}$ is important.\\

\subsection{Conclusion.}

Finally, we would like to give a conclusion. From the numerical analysis results of the two cases, if the components of the matrix contain $\varepsilon$, then it is assumed that the more the position information of the matrix is known, the shorter the calculation time. In addition, it can be considered that it is more important to reduce the number of calculations than to use information that is already known. In the sense of realizing the importance of this, we do not think that numerical results for Theorem $\ref{main}$ are wasted. Based on the above considerations, it can be concluded that reducing the number of calculations, that is, how to reduce the wasteful act of adding $\varepsilon$ and , is important in numerical linear algebra.
\begin{equation}[\text{Wasteful Act}] \quad
\varepsilon + * = \varepsilon
\end{equation}
From a theoretical point of view, direct substitution calculations using graph theory may also be important. However, in this case, only the diagonal component was used, and since it was conditional, its importance could not be demonstrated. On the other hand, We hope that this paper will recognize the value and interest of demonstrating its importance.

\newpage

\section{Bounus : Jet Black Theorem , Elementary cellular automaton and Oshida Rule}

At first, we introduce the definition of perfect for the matrix of max-plus algebra.

\begin{defi}[Perfect]
\it{
Let $n$ be the non-negative integer. For} $A \in \mathbb{R}^{n \times n}_{\text{max}}$, \it{if there exists
\begin{equation}
m:=\min \{ p \in \mathbb{Z}_{>1} \mid \text{$[A^{\otimes p}]_{i,j} \neq \varepsilon$ for all $(i,j) \in [1,n]^2 $}  \}, \nonumber
\end{equation}
then we say that $A$ is $m$-perfect.
}
\end{defi}

For non-negative integer $m$ with $m>1$, honest matrix $A \in \mathbb{R}^{(2m+1) \times (2m+1)}_{\text{max}}$ is $2m$-perfect by the following result.

\begin{theo}\label{PerfectHonest}
\it{
We have
\begin{equation}\label{Per_Per}
\min \{ p \in \mathbb{Z}_{>1} \mid \text{$[A^{\otimes p}]_{i,j} \neq \varepsilon$ for all $(i,j) \in [1,2m+1]^2 $}  \}=2m.
\end{equation}
}
\end{theo}

\begin{proof}At first, we prove that the following inequality holds.
\begin{equation}\label{Per_1}
\min \{ p \in \mathbb{Z}_{>1} \mid \text{$[A^{\otimes p}]_{i,j} \neq \varepsilon$ for all $(i,j) \in [1,2m+1]^2 $}  \} \leq 2m.
\end{equation}
On the other word, we prove that we have $[A^{\otimes 2m}]_{i,j} \neq \varepsilon$ for all $(i,j) \in [1,2m+1]^2$.\par

Let $e$ be the first row vector of $A^{\otimes m}$. First, $[A^{\otimes 2m}]_{1,1}$ satisfies
\begin{eqnarray}\label{p_1,1}
[A^{\otimes 2m}]_{1,1}
&=& [A^{\otimes m} \otimes A^{\otimes m}]_{1,1} \nonumber \\
&=&\max_k  [A^{\otimes m}]_{1,k} + [A^{\otimes m}]_{k,1} \nonumber \\
&=&\max \{ [A^{\otimes m}]_{1,m+1} + [A^{\otimes m}]_{m+1,1} , \max_{k \neq m+1} [A^{\otimes m}]_{1,k} + [A^{\otimes m}]_{k,1} \} \nonumber \\
&\geq&[A^{\otimes m}]_{1,m+1} + [A^{\otimes m}]_{m+1,1}.
\end{eqnarray}
Since
\begin{eqnarray}\label{p_im}
f_{m-1}(e)
&=&(\underbrace{\varepsilon,\cdots,\varepsilon}_{m},[f_{m-1}(e)]_{m+1},[f_{m-1}(e)]_{m+2},[f_{m-1}(e)]_{m+3},\cdots,[f_{m-1}(e)]_{2m+1}) \nonumber \\
&=&(\underbrace{\varepsilon,\cdots,\varepsilon}_{m},[A^{\otimes m}]_{m,m+1},[A^{\otimes m}]_{m,m+2},[A^{\otimes m}]_{m,m+3},\cdots,[A^{\otimes m}]_{m,2m+1}) \nonumber \\
&\neq&(\varepsilon,\cdots,\varepsilon),
\end{eqnarray}
we have $[f_m(e)]_1=[A^{\otimes m}]_{m+1,1} \neq \varepsilon$. Hence, we have
\begin{equation}\label{p_1,2}
[A^{\otimes m}]_{1,m+1} + [A^{\otimes m}]_{m+1,1} \neq \varepsilon.
\end{equation}
Thus, we have $[A^{\otimes 2m}]_{1,1} \neq \varepsilon$ by $(\ref{p_1,1})$ and $(\ref{p_1,2})$

Next, $[A^{\otimes 2m}]_{1,2}$ satisfies
\begin{eqnarray}\label{p_2,1}
[A^{\otimes 2m}]_{1,2}
&=&[A^{\otimes m} \otimes A^{\otimes m}]_{1,2} \nonumber \\
&=&\max_k [A^{\otimes m}]_{1,k} + [A^{\otimes m}]_{k,2} \nonumber \\
&=&\max \{ [A^{\otimes m}]_{1,m+2} + [A^{\otimes m}]_{m+2,2} , \max_{k \neq m+2}  [A^{\otimes m}]_{1,k} + [A^{\otimes m}]_{k,2} \} \nonumber \\
&\geq&[A^{\otimes m}]_{1,m+2} + [A^{\otimes m}]_{m+2,2}.
\end{eqnarray}
By $(\ref{p_im})$, we have $[f_{m+1}(e)]_2=[A^{\otimes m}]_{m+2,2} \neq \varepsilon$. By Corollay $\ref{CPNocoro}$, we have 
\begin{eqnarray}
[A^{\otimes m}]_{1, \color{blue} 2(m+1)-(m-1)-1}
&=&[A^{\otimes m}]_{1, \color{blue} 2m+2-m+1-1} \nonumber \\
&=&[A^{\otimes m}]_{1, \color{blue} m+2} \nonumber \\
&\neq& \varepsilon. \nonumber 
\end{eqnarray}
Hence,  we have 
\begin{equation}
[A^{\otimes m}]_{1,m+2} + [A^{\otimes m}]_{m+2,2} \neq \varepsilon. \nonumber
\end{equation}
Thus, we have $[A^{\otimes 2m}]_{1,2} \neq \varepsilon$ by $(\ref{p_2,1})$.

Next, $[A^{\otimes 2m}]_{1,j}$ satisfies the following inequality for $j \in [3,m+2]$.
\begin{eqnarray}\label{p_n,1}
[A^{\otimes 2m}]_{1,j}
&=&[A^{\otimes m} \otimes A^{\otimes m}]_{1,j} \nonumber \\
&=&\max_{k}  [A^{\otimes m}]_{1,k} + [A^{\otimes m}]_{k,j} \nonumber \\
&=&\{ \varepsilon + [A^{\otimes m}]_{2,j} , \max_{2 \leq k \leq 2m+1}  [A^{\otimes m}]_{1,k} + [A^{\otimes m}]_{k,j} \} \nonumber \\
&=&\max_{2 \leq k \leq 2m+1} [A^{\otimes m}]_{1,k} + [A^{\otimes m}]_{k,j} \nonumber \\
&=&\max \{ [A^{\otimes m}]_{1,2} + [A^{\otimes m}]_{2,j} , \max_{2 \leq k \leq 2m+1} [A^{\otimes m}]_{1,k} + [A^{\otimes m}]_{k,j} \} \nonumber \\
&\geq&[A^{\otimes m}]_{1,2} + [A^{\otimes m}]_{2,j} \nonumber \\
&=&[A^{\otimes m}]_{1,2} + [f_1(e)]_{j}.
\end{eqnarray}
Since $f_1(e)$ satisfies
\begin{equation}
f_1(e)
=(\varepsilon,\varepsilon,[f_1(e)]_3,\cdots,f_1(e)_{m+3},\varepsilon,\cdots,\varepsilon)
\neq (\varepsilon,\cdots,\varepsilon), \nonumber
\end{equation}
we have $[f_1(e)]_{j} \neq \varepsilon$ for $j \in [3,m+2]$. Hence, we have 
\begin{equation}
[A^{\otimes m}]_{1,2} + [f_1(e)]_{j} \neq \varepsilon \nonumber
\end{equation}
by $[A^{\otimes m}]_{1,2} \neq \varepsilon$. Thus, we have $[A^{\otimes 2m}]_{1,j} \neq \varepsilon$ for $j \in [3,m+2]$ by $(\ref{p_n,1})$.\par

Final, $[A^{\otimes 2m}]_{1,j}$ satisfies the following inequality
for $r:=m+1$ and $j \geq m+3$.
\begin{eqnarray}\label{p_3,1}
[A^{\otimes 2m}]_{1,j}
&=&[A^{\otimes m} \otimes A^{\otimes m}]_{1,j} \nonumber \\
&=&\max_{k}  [A^{\otimes m}]_{1,k} + [A^{\otimes m}]_{k,j} \nonumber \\
&=&\{ \varepsilon + [A^{\otimes m}]_{2,j} , \max_{2 \leq k \leq 2m+1}  [A^{\otimes m}]_{1,k} + [A^{\otimes m}]_{k,j} \} \nonumber \\
&=&\max_{2 \leq k \leq 2m+1} [A^{\otimes m}]_{1,k} + [A^{\otimes m}]_{k,j} \nonumber \\
&\geq&\max_{2 \leq k \leq m} [A^{\otimes m}]_{1,k} + [A^{\otimes m}]_{k,j} \nonumber \\
&=&\max \{ [A^{\otimes m}]_{1,j-(m+1)} + [A^{\otimes m}]_{j-(m+1),j} , \max_{\substack{2 \leq k \leq m, \\ k \neq j-(m+1)}} [A^{\otimes m}]_{1,k} + [A^{\otimes m}]_{k,j} \} \nonumber \\
&\geq&[A^{\otimes m}]_{1,j-r} + [A^{\otimes m}]_{j-r,j}
\end{eqnarray}
Since it follows from $m+3 \leq j \leq 2m+1$ that $j-r$ satisfies $2 \leq j-r \leq m$, we have 
\begin{equation}\label{p_3,2}
[A^{\otimes m}]_{1,j-r}=[e]_{j-r} \neq \varepsilon
\end{equation}
for $j-r \in [2,m] \subset [2,m+2]$ by Corollary $\ref{CPNocoro}$. Moreover, since  block matrixs
\begin{equation}
C=
\begin{pmatrix}
[A^{\otimes m}]_{1,m+3} & [A^{\otimes m}]_{1,m+4} & [A^{\otimes m}]_{1,m+5} & \cdots & [A^{\otimes m}]_{1,2m} & [A^{\otimes m}]_{1,2m+1} \\
[A^{\otimes m}]_{2,m+3} & [A^{\otimes m}]_{2,m+4} & [A^{\otimes m}]_{2,m+5} & \cdots & [A^{\otimes m}]_{2,2m} & [A^{\otimes m}]_{2,2m+1} \\
[A^{\otimes m}]_{3,m+3} & [A^{\otimes m}]_{3,m+4} & [A^{\otimes m}]_{3,m+5} & \cdots & [A^{\otimes m}]_{3,2m} & [A^{\otimes m}]_{3,2m+1} \\
[A^{\otimes m}]_{4,m+3} & [A^{\otimes m}]_{4,m+4} & [A^{\otimes m}]_{4,m+5} & \cdots & [A^{\otimes m}]_{4,2m} & [A^{\otimes m}]_{4,2m+1} \\
\vdots & \vdots & \vdots & \ddots & \vdots & \vdots \\
[A^{\otimes m}]_{m-1,m+3} & [A^{\otimes m}]_{m-1,m+4} & [A^{\otimes m}]_{m-1,m+5} & \cdots & [A^{\otimes m}]_{m-1,2m} & [A^{\otimes m}]_{m-1,2m+1} \\  
[A^{\otimes m}]_{m,m+3} & [A^{\otimes m}]_{m,m+4} & [A^{\otimes m}]_{m,m+5} & \cdots & [A^{\otimes m}]_{m,2m} & [A^{\otimes m}]_{m,2m+1} \\  
\end{pmatrix}
\nonumber
\end{equation}
of $A^{\otimes m}$ satisfies 
\begin{eqnarray}
&&C \nonumber \\
&=&
\begin{pmatrix}
\varepsilon & \varepsilon & \varepsilon & \cdots & \varepsilon & \varepsilon \\
[A^{\otimes m}]_{\color{blue}2,m+3} & \varepsilon & \varepsilon & \cdots & \varepsilon & \varepsilon \\
[A^{\otimes m}]_{3,m+4} & [A^{\otimes m}]_{\color{blue}3,m+4} & \varepsilon & \cdots & \varepsilon & \varepsilon \\
[A^{\otimes m}]_{4,m+5} & [A^{\otimes m}]_{4,m+5} & [A^{\otimes m}]_{\color{blue}4,m+5} & \cdots & \varepsilon & \varepsilon \\
\vdots & \vdots & \vdots & \ddots & \vdots & \vdots \\
[A^{\otimes m}]_{m-1,m+3} & [A^{\otimes m}]_{m-1,m+4} & [A^{\otimes m}]_{m-1,m+5} & \cdots & [A^{\otimes m}]_{\color{blue}m-1,2m} & \varepsilon \\  
[A^{\otimes m}]_{m,m+3} & [A^{\otimes m}]_{m,m+4} & [A^{\otimes m}]_{m,m+5} & \cdots & [A^{\otimes m}]_{m,2m} & [A^{\otimes m}]_{\color{blue}m,2m+1} \\
\end{pmatrix}
\nonumber \\
&=&
\begin{pmatrix}
\varepsilon & \varepsilon & \varepsilon & \cdots & \varepsilon & \varepsilon \\
\color{blue}[A^{\otimes m}]_{(m+3)-r,m+3} & \varepsilon & \varepsilon & \cdots & \varepsilon & \varepsilon \\
[A^{\otimes m}]_{3,m+3} & \color{blue}[A^{\otimes m}]_{(m+4)-r,m+4} & \varepsilon & \cdots & \varepsilon & \varepsilon \\
[A^{\otimes m}]_{4,m+3} & [A^{\otimes m}]_{4,m+4} & \color{blue}[A^{\otimes m}]_{(m+5)-r,m+5} & \cdots & \varepsilon & \varepsilon \\
\vdots & \vdots & \vdots & \ddots & \vdots & \vdots \\
[A^{\otimes m}]_{m-1,m+3} & [A^{\otimes m}]_{m-1,m+4} & [A^{\otimes m}]_{m-1,m+5} & \cdots & \color{blue}[A^{\otimes m}]_{2m-r,2m} & \varepsilon \\  
[A^{\otimes m}]_{m,m+3} & [A^{\otimes m}]_{m,m+4} & [A^{\otimes m}]_{m,m+5} & \cdots & [A^{\otimes m}]_{m,2m} & \color{blue}[A^{\otimes m}]_{(2m+1)-r,2m+1} \\
\end{pmatrix}
\nonumber \\
&\neq&
\begin{pmatrix}
\varepsilon & \varepsilon & \varepsilon & \cdots & \varepsilon & \varepsilon \\
\color{blue}\varepsilon & \varepsilon & \varepsilon & \cdots & \varepsilon & \varepsilon \\
\varepsilon & \color{blue}\varepsilon & \varepsilon & \cdots & \varepsilon & \varepsilon \\
\varepsilon & \varepsilon & \color{blue}\varepsilon & \cdots & \varepsilon & \varepsilon \\
\vdots & \vdots & \vdots & \ddots & \vdots & \vdots \\
\varepsilon & \varepsilon & \varepsilon & \cdots & \color{blue}\varepsilon & \varepsilon \\  
\varepsilon & \varepsilon & \varepsilon & \cdots & \varepsilon & \color{blue}\varepsilon \\
\end{pmatrix}
\nonumber
\end{eqnarray}
by the vector
\begin{equation}
[f_0(e)]=(\varepsilon,[A^{\otimes m}]_{1,2},\cdots,[A^{\otimes m}]_{2,m+1}, \color{blue} [A^{\otimes m}]_{2,m+2} \color{black} ,\color{red}\varepsilon,\cdots,\varepsilon \color{black}) \nonumber
\end{equation}
and
\begin{equation}
[f_1(e)]=(\varepsilon,\varepsilon,[A^{\otimes m}]_{2,3},\cdots, \color{blue} [A^{\otimes m}]_{2,m+2} \color{black} ,\color{red}[A^{\otimes m}]_{2,m+3},\varepsilon,\cdots,\varepsilon \color{black}), \nonumber
\end{equation}
we have 
\begin{equation}\label{p_3,3}
[A^{\otimes m}]_{j-r,j} \neq \varepsilon
\end{equation}
for $(j-r,j) \in [2,m] \times [m+3,2m+1]$.\par

Hence, we obtain the following result by $(\ref{p_3,2})$ and $(\ref{p_3,3})$.
\begin{equation}\label{p_3,4}
[A^{\otimes m}]_{1,j-(m+1)} + [A^{\otimes m}]_{j-(m+1),j} \neq \varepsilon.
\end{equation}

Thus, we have $[A^{\otimes 2m}]_{1,j}$ for $j \geq m+3$ by $(\ref{p_3,1})$ and $(\ref{p_3,4})$. From the above discussion, we have $[A^{\otimes 2m}]_{i,j} \neq \varepsilon$ for all $(i,j) \in [1,2m+1]^2$. Therefore, the inequality $(\ref{Per_1})$ holds.\\

Second, we prove that the following inequality holds.
\begin{equation}\label{Per_2}
2m-1 < \min \{ p \in \mathbb{Z}_{>1} \mid \text{$[A^{\otimes p}]_{i,j} \neq \varepsilon$ for all $(i,j) \in [1,2m+1]^2 $}  \}.
\end{equation}
To prove that $(\ref{Per_2})$ holds, it is sufficient to show that we have 
\begin{equation}\label{Per_22}
[A^{\otimes (2m-1)}]_{1,4}= \varepsilon.
\end{equation}
Let $e$ be the first row vector of $A^{\otimes (m-1)}$. The vector $e$ is constructed as follows.
\begin{equation}
[e]_j
\begin{cases}
\neq \varepsilon & \text{for $j \in [4,m+3]$}, \\
=\varepsilon & \text{for $j \in [1,3] \cup [m+4,2m+1]$}.
\end{cases}
\nonumber 
\end{equation}
Since 
\begin{eqnarray}
\# [m+4,2m+1] 
&=&\# \{ m+4 , m+5 , \cdots , 2m+1 \} \nonumber \\
&=&(2m+1)-(m+4)+1 \nonumber \\
&=&m-2, \nonumber 
\end{eqnarray}
we have 
\begin{equation}\label{Zureteru1}
[f_1(e)]_{4}=\varepsilon
\end{equation}
and
\begin{equation}\label{Zureteru2}
[f_{m-2}(e)]_j = \varepsilon
\end{equation}
for $j \in [1,m+1]$. By $(\ref{Zureteru1})$ and $(\ref{Zureteru2})$, we have $[f_i(e)]_4=\varepsilon$ for $i=1,2,\cdots,m+1=(m-2)+3$. Hence, $A^{\otimes (m-1)}$ satisfies
\begin{equation}\label{Zureteru3}
[A^{\otimes (m-1)}]_{i+1,4}=[f_i(e)]_4=\varepsilon
\end{equation}
for $i=1,2,\cdots,m+1=(m-2)+3$. By $(\ref{Zureteru3})$ and
and
\begin{equation}
[A^{\otimes m}]_{1,j}
\begin{cases}
\neq \varepsilon & \text{for $j \in [2,m+2]$}, \\
=\varepsilon & \text{for $j \in \{ 1 \} \cup [m+3,2m+1]$}.
\end{cases}
, \nonumber 
\end{equation}
we have
\begin{eqnarray}
[A^{\otimes (2m-1)}]_{1,4}
&=&[A^{\otimes m} \otimes A^{\otimes (m-1)}]_{1,4} \nonumber \\
&=&\max_{k} [A^{\otimes m}]_{1,k} + [A^{\otimes (m-1)}]_{k,4} \nonumber \\
&=&\max \{ \color{red} [A^{\otimes m}]_{1,1} \color{black} + \color{red} [A^{\otimes (m-1)}]_{1,4} \color{black} , \max_{k \neq 1} [A^{\otimes m}]_{1,k} + [A^{\otimes (m-1)}]_{k,4} \} \nonumber \\
&=&\max \{ \color{red} \varepsilon \color{black} , \max_{k \neq 1} [A^{\otimes m}]_{1,k} + [A^{\otimes (m-1)}]_{k,4} \} \nonumber \\
&=&\max_{k \neq 1} [A^{\otimes m}]_{1,k} + [A^{\otimes (m-1)}]_{k,4} \nonumber \\
&=&\max \{ \max_{k \in [2,m+2]} [A^{\otimes m}]_{1,k} + \color{red} [A^{\otimes (m-1)}]_{k,4} \color{black} , \max_{k \in [m+3,2m+1]} [A^{\otimes m}]_{1,k} + [A^{\otimes (m-1)}]_{k,4} \}  \nonumber \\
&=&\max \{ \color{red} \varepsilon \color{black} , \max_{k \in [m+3,2m+1]} [A^{\otimes m}]_{1,k} + [A^{\otimes (m-1)}]_{k,4} \}  \nonumber \\
&=&\max_{k \in [m+3,2m+1]} \color{red} [A^{\otimes m}]_{1,k} \color{black} + [A^{\otimes (m-1)}]_{k,4}.
\end{eqnarray}
Thus, $(\ref{Per_22})$ holds. Therefore, it follows from $(\ref{Per_1})$ and $(\ref{Per_2})$ that we have $(\ref{Per_Per})$
\end{proof}

As application of Theorem $\ref{PerfectHonest}$, we introduce ''Jet Black Theorem''. Let
\begin{eqnarray}
E(A)=\{ f_i(e) \mid \text{The vector $e$ be the first row vector of $A$ and $i \in [0,2m]$.} \} \nonumber
\end{eqnarray}
for honest matrix $A \in \mathbb{R}^{2m+1 \times 2m+1}_{\text{max}}$, where $m$ be non-negative integer with $m>1$. Let 
\begin{eqnarray}
F_{2m+1}=\{ u \in \mathbb{F}_2[X] \mid \text{$x^{2m+i}:=x^i$ for $i=2,\cdots,2m+1$} \}. \nonumber
\end{eqnarray}
Let $h:\mathbb{R}^{2m+1}_{\text{max}} \to F_{2m+1}$ be the function defined by
\begin{eqnarray}
h(u)=\sum_{i=1}^{2m+1} \text{Exist}([u]_i)x^{\sigma([u]_i)},
\end{eqnarray}
where
\begin{equation}
\text{Exist}([u]_i)=
\begin{cases}
1 & \text{if $[u]_i \neq \varepsilon$}, \\
0 & \text{if $[u]_i = \varepsilon$}.
\end{cases}
\end{equation}
and
\begin{eqnarray}
\sigma([u]_i)=2m+1-(i-1).
\end{eqnarray}

Then, we have the following result.

\begin{theo}\label{OJBT_1997}[Oshida's Jet Black Theorem]
\it{
Let $f_i:\mathbb{R}^{2m+1}_{\text{max}} \to \mathbb{R}^{2m+1}_{\text{max}}$ be the function which satisfies
\begin{equation}
f_i(u):=u \otimes A \nonumber
\end{equation}
for $i=1,2,\cdots,m-1$, where $u \in \mathbb{R}^{2m+1}_{\text{max}}$. Let $g_i:F_{2m+1} \to F_{2m+1}$ be the function which satisfies
\begin{equation}
g_i(\color{red} x^{m_1} + x^{m_1-1} + \cdots + x^{m_2+1} + x^{m_2} \color{black}{):= (} \color{blue}{x^{m_1+2} + x^{m_1+1}} ) \color{black} + ( \color{red} x^{m_1} + x^{m_1-1} + \cdots + x^{m_2+1} + x^{m_2} \color{black} ) + \color{blue} x^{m_2} \nonumber
\end{equation}
for $i=1,2,\cdots,m-1$, where $m_1,m_2 \in \mathbb{Z}_{>0}$ with $m_1>m_2$.\par
Then, for the following diagram ''Jet Black$-245$ Imitation Diagram'', we have 
\begin{equation}
h \circ f_i = g_i \circ h
\end{equation}
for $i=1,2,\cdots,m-1$ and we have $h \circ f^{m}_1=g^{m}_1 \circ h$. On the other hand, we can deduce that $A$ is 2m-perfect through the action $g_1$.
\[
  \begin{CD}
     E(A) \\
     @V{p}VV \\
     \{e\} @>{f_1}>> \mathbb{R}^{2m+1}_{\text{max}} @>{f_2}>> \mathbb{R}^{2m+1}_{\text{max}} @>{f_3}>> \cdots @>{f_{m-2}}>> \mathbb{R}^{2m+1}_{\text{max}} @>{f_{m-1}}>> \mathbb{R}^{2m+1}_{\text{max}} @>{f^{m}_1}>> \mathbb{R}^{2m+1}_{\text{max}} \\
  @V{h}VV  @V{h}VV  @V{h}VV  @.  @V{h}VV @V{h}VV @V{h}VV  \\
     F_{2m+1}   @>{g_1}>>  F_{2m+1} @>{g_2}>> F_{2m+1} @>{g_3}>> \cdots @>{g_{m-2}}>> F_{2m+1} @>{g_{m-1}}>> F_{2m+1} @>{g^{m}_1}>> \Bigg \{ \displaystyle \sum_{i=1}^{2m+1} x^i \Bigg \} \\
     @. @. @. @. @. @. @V{q}VV \\
     @. @. @. @. @. @. E(A^{\otimes 2m}) 
  \end{CD}
\]
}
\end{theo}

We call $q \circ g^{2m-1}_1 \circ h \circ p$ Jet Black Scheme.

\begin{proof}[Proof of Theorem $\ref{OJBT_1997}$]At first, we assume that $m=2$. Then, we have 
\begin{equation}
e=(0,0,0,[e]_4,[e]_5),
\end{equation}
where $[e]_4<0$ and $[e]_5>0$. By proof of Proposition $\ref{CPNo}$, we have the following diagram.
\[
  \begin{CD}
     \{e\} @>{f_1}>> \{ e \otimes A \} \\
  @V{h}VV  @V{h}VV \\
  \{ x^2+x \}   @. \{ x^4+x^3+x^2 \}
  \end{CD}
\]
Hence we have $h \circ f_1=g_1 \circ h$ by
\begin{eqnarray}
g_1(x^2+x)=(x^4+x^3)+(x^2+x)+x=x^4+x^3+x^2.
\end{eqnarray}
Since
\begin{eqnarray}
g_1(g_1(x^4+x^3+x^2))
&=&g_1((x^6+x^5)+(x^4+x^3+x^2)+x^2) \nonumber \\
&=&g_1(x^5+x^4+x^3+x) \nonumber \\
&=&(x^7+x^6)+(x^5+x^4+x^3+x)+x \nonumber \\
&=&x^5+x^4+x^3+x^2+x, \nonumber
\end{eqnarray}
we have the following diagram by Theorem $5.2$.
\[
  \begin{CD}
     \{ e \otimes A \} @>{f^2_1}>> \{ (e \otimes A) \otimes A^{\otimes 2} = e \otimes A^{\otimes 3} \} \\
  @V{h}VV  @V{h}VV \\
  \{ x^4+x^3+x^2 \}   @>{g^2_1}>> \{ x^5+x^4+x^3+x^2+x \}
  \end{CD}
\]
Hence, we  $h \circ f^2_1=g^2_1 \circ h$. Thus, Theorem $\ref{OJBT_1997}$ holds for $m=2$.\par

Next, we assume that Theorem $\ref{OJBT_1997}$ holds for $m>2$. Let
\begin{equation}
\beta_1(k)=2(m+1)-2k-3
\end{equation}
and
\begin{equation}
\beta_2(k)=2(m+1)-k.
\end{equation}
Then, we have the following diagram by Proposition $\ref{CPNo}$ and Corollary $\ref{CPNocoro}$.
\[
  \begin{CD}
     \{ e \} @>{f_1}>> \{ e \otimes A \} @>{f_2}>> \{ e \otimes A^{\otimes 2} \} @>{f_3}>> \cdots @>{f_{m-1}}>> \{ e \otimes A^{\otimes (m-1)} \} \\
  @V{h}VV  @V{h}VV @V{h}VV @. @V{h}VV \\
  \{ x^2+x \} @. \{ h(e \otimes A) \} @. \{ h(e \otimes A^{\otimes 2}) \} @. @. \{ h(e \otimes A^{\otimes (m-1)}) \}
  \end{CD},
\]
where the polynomial
\begin{equation}
h(e \otimes A^{k})
=\sum_{i=\sigma([e \otimes A^{\otimes k}]_{\beta_2(k)-1})}^{\sigma([e \otimes A^{\otimes k}]_{\beta_1(k)+1})} x^i. \nonumber
\end{equation}
To prove that Theorem $\ref{OJBT_1997}$ holds for $m+1$, we prove below equations at first.
\begin{eqnarray}\label{Uruf}
h(e \otimes A^{\otimes n})
=g_n \circ h (e \otimes A^{\otimes (n-1)})
\end{eqnarray}
If $n=1$, then equations $(\ref{Uruf})$ holds by the following calculation result.
\begin{eqnarray}
h(e \otimes A) 
&=&x^4+x^3+x^2 \nonumber \\
&=&(x^4+x^3)+(x^2+x)+x \nonumber \\
&=&g_1(x^2+x) \nonumber \\
&=&g_1 \circ h (e)
\end{eqnarray}
Next, for $1 \leq n < m-1$, we assume that equations $(\ref{Uruf})$ holds. The polynomial $h(e \otimes A^{\otimes n+1})$ can rewrite as 
\begin{eqnarray}\label{OJBT_Im1}
h(e \otimes A^{\otimes n+1}) 
=\sum_{i=\sigma([e \otimes A^{\otimes {n+1}}]_{\beta_2(n+1)-1})}^{\sigma([e \otimes A^{\otimes {n+1}}]_{\beta_1(n+1)+1})} x^i
=\sum_{i=n+2}^{2n+4} x^i
\end{eqnarray}
becaous of the following calculation of $\sigma([e \otimes A^{\otimes {n+1}}]_{\beta_2(n+1)-1})$ and $\sigma([e \otimes A^{\otimes {n+1}}]_{\beta_1(n+1)+1})$.
\begin{eqnarray}
\sigma([e \otimes A^{\otimes {n+1}}]_{\beta_2(n+1)-1})
&=&2m+1-((\beta_2(n+1)-1)-1) \nonumber \\
&=&2m+1-(\beta_2(n+1)-2) \nonumber \\
&=&2m+1-(2(m+1)-(n+1)-2) \nonumber \\
&=&2m+1-((2m+1)-(n+1)-1) \nonumber \\
&=&2m+1-(2m+1)+(n+1)+1 \nonumber \\
&=&n+2 \nonumber
\end{eqnarray}
\begin{eqnarray}
\sigma([e \otimes A^{\otimes {n+1}}]_{\beta_1(n+1)+1})
&=&2m+1-((\beta_1(n+1)+1)-1) \nonumber \\
&=&2m+1-\beta_1(n+1) \nonumber \\
&=&2m+1-(2(m+1)-2(n+1)-3) \nonumber \\
&=&2m+1-((2m+1)-2(n+1)-2) \nonumber \\
&=&2m+1-(2m+1)+2(n+1)+2  \nonumber \\
&=&2n+4. \nonumber
\end{eqnarray}
Hence, the polynomial $(\ref{OJBT_Im1})$ can calculate as follows.
\begin{eqnarray}
h(e \otimes A^{\otimes n+1})
=\sum_{i=n+2}^{2n+4} x^i
&=&(x^{2n+4} + x^{2n+3}) + \sum_{i=n+2}^{2n+2} x^i \nonumber \\
&=&(x^{2n+4} + x^{2n+3}) + \Biggr ( \sum_{i=n+2}^{2n+2} x^i \Biggr ) + x^{n+1} + x^{n+1} \nonumber \\
&=&(x^{2n+4} + x^{2n+3}) + \Biggr ( \sum_{i=n+1}^{2n+2} x^i \Biggr ) + x^{n+1} \nonumber \\
&=&g_{n+1} \Biggr ( \sum_{i=n+1}^{2n+2} x^i \Biggr ) \nonumber \\
&=&g_{n+1} \circ h (e \otimes A^{\otimes n}). \nonumber
\end{eqnarray}
Thus, the equation $(\ref{Uruf})$ holds.\par
Next, we can calculate $h \circ f^{m}_1(e \otimes A^{\otimes (m-1)})$ as follows by Theorem $5.2$.
\begin{equation}
h \circ f^{m}_1(e \otimes A^{\otimes (m-1)})
=h ( f^{m}_1(e \otimes A^{\otimes (m-1)}) )
=h ( \color{red} e \otimes A^{\otimes (2m-1)} \color{black} )
=\sum_{i=1}^{2m+1} x^i. \nonumber
\end{equation}
Note that the set $\{ e \otimes A^{\otimes (2m-1)} \}$ is subset of  the set of all row vectors that make up the matrix $A^{\otimes 2m}$. Hence, we have
\begin{eqnarray}
g^m_1 \circ h (e \otimes A^{\otimes (m-1)})
=g^m_1 ( h (e \otimes A^{\otimes (m-1)}) )
&=&g^m_1 \Biggr ( \sum_{i=m}^{2m} x^i \Biggr ) \nonumber \\
&=&g_1^{ \color{blue} m \color{black} } \Biggr ( x^{2m} + \color{blue}  \sum_{i=m}^{ 2m-1 } \color{black} x^i \Biggr ) \nonumber \\
&=&\Biggr ( \color{blue}  \sum_{i=1}^{m} \color{black} x^{2m+(2i-1)} + x^{2m+2i} \Biggr ) + \Biggr ( x^{2m} + \sum_{i=m}^{2m-1} x^i \Biggr ) + \Biggr ( \color{blue}  \sum_{i=m}^{2m-1} \color{black} x^i \Biggr ) \nonumber \\
&=&\Biggr ( \sum_{i=2m+1}^{4m} x^{i} \Biggr ) + \Biggr ( x^{2m} + \color{red} \sum_{i=m}^{2m-1} x^i \color{black} \Biggr ) + \Biggr ( \color{red} \sum_{i=m}^{2m-1} x^i \color{black} \Biggr ) \nonumber \\
&=&x^{2m}  + \sum_{i=2m+1}^{4m} x^{i} \nonumber \\
&=&x^{2m+1} + x^{2m} + \color{blue} \sum_{i=2m+2}^{4m} x^{ i } \nonumber \\
&=&x^{2m+1} + x^{2m} + \color{blue} \sum_{i=1}^{2m-1} x^{2m+1+i} \nonumber \\
&=&x^{2m+1} + x^{2m} + \color{blue} \sum_{i=1}^{2m-1} x^{i} \nonumber \\
&=&\sum_{i=1}^{2m+1} x^{i} \nonumber \\
&=&h \circ f^{m}_1(e \otimes A^{\otimes (m-1)}). \nonumber
\end{eqnarray}
Thus, Theorem $\ref{OJBT_1997}$ holds for $m+1$. Therefore, Theorem $\ref{OJBT_1997}$ holds for all $m>1$.
\end{proof}

Finally, we introduce the importance conjecture of a One-Dimensional Cellular Automaton.

\begin{equation}
\text{CA}(t)=\{ (U_{t,1},\cdots,U_{t,2m+1}) \mid \text{$U_{t,i} \in \mathbb{F}_2$ for $i=1,2,\cdots,2m+1$} \} \nonumber
\end{equation}

For the vector $\Lambda \in \text{CA}(t)$, let $S:\text{CA}(t) \to \text{CA}(t)$ be the function defined as follows.
\begin{equation}
[S(\Lambda)]_j=
\begin{cases}
[\Lambda]_j + \bar{1} & \text{if $j=\max \{ i \mid [\Lambda]_i \neq 0 \}$}, \\
[\Lambda]_j & \text{otherwise}.
\end{cases}
\end{equation}
Let $L:\text{CA}(t) \to \text{CA}(t)$ be the function defined by
\begin{equation}
[L(\Lambda)]_j=
\begin{cases}
[\Lambda]_j + \bar{1} & \text{if $j=\min \{ i \mid [\Lambda]_i \neq 0 \}+1$}, \\
[\Lambda]_j + \bar{1} & \text{if $j=\min \{ i \mid [\Lambda]_i \neq 0 \}+2$}, \\
[\Lambda]_j & \text{otherwise}.
\end{cases},
\end{equation}
where $[L(\Lambda)]_j:=[L(\Lambda)]_{j-(2m+1)}$ for $j>2m+1$. Let $h':F_{2m+1} \to \text{CA}(t)$ be the function by defined by
\begin{equation}
\Biggr [h' \Biggr ( \sum_{i=1}^{2m+1} c_i x^i \Biggr ) \Biggr ]_j
=
c_j
.
\end{equation}
Then, we have 
\begin{equation}\label{Oshida-Rule Before}
h' \circ h \circ \underbrace{ f_1 \circ \cdots \circ f_1 }_{2m-1}(e)=\underbrace{ H_{245} \circ \cdots \circ H_{245} }_{2m-1} \circ h'  \circ h (e)
\end{equation}
by the following function.
\begin{equation}\label{Oshida-Rule}[Oshida-Rule] \quad
H_{245}=L \circ S
\end{equation}

About the above rules, we can consider the following two conjectures.

\begin{conj}[Initial value problem of Oshida-Rule]\it{
Let $n$ be the non-negative intger. Then, for the vector $u \in \mathbb{R}^{n}_{\text{max}}$, we have
\begin{equation}
h' \circ h \circ \underbrace{ f_1 \circ \cdots \circ f_1 }_{2m-1}(u) \neq \underbrace{ H_{245} \circ \cdots \circ H_{245} }_{2m-1} \circ h'  \circ h (u)
\end{equation}
when $u \neq e$. In other words, the following equality holds.
\begin{equation}
\{ u \in \mathbb{R}^{n}_{\text{max}} \mid \text{$u$ satisfies $h' \circ h \circ \underbrace{ f_1 \circ \cdots \circ f_1 }_{2m-1}(u) = \underbrace{ H_{245} \circ \cdots \circ H_{245} }_{2m-1} \circ h'  \circ h (u)$.} \} = \{ e \}. \nonumber
\end{equation}
}
\end{conj}

\begin{conj}\label{FNOR}[Non-commutativity and Linguization of Oshida-Rule]\it{
Let $A$ be the honest matrix for $m \in \mathbb{Z}_{>1}$. Then, we can not  prove that $(\ref{Oshida-Rule Before})$ holds by matrixs $A$. On other words, we can not prove that we can construct $H_{245}$ by using matrixs $A$.
}
\end{conj}

I believe in positive for above two conjectures. Note that it is very obvious that we can not prove that we can construct $H_{245}$ using matrixs $g_1$.

\end{document}